\documentclass[a4paper,11pt,leqno]{amsart}
\usepackage[utf8]{inputenc}
\usepackage[english]{babel}
\usepackage{amsmath}
\usepackage{amsfonts}
\usepackage{enumerate}
\usepackage{amssymb}
\usepackage{amsmath,amssymb,array,euscript,amsfonts}
\usepackage{graphicx,epsfig,color}
\usepackage[usenames,dvipsnames]{xcolor}
\usepackage{rotating}
\usepackage{bbm}
\usepackage{color}
\usepackage{hyperref}
\usepackage{tikz}
\usepackage{appendix}
\usepackage{chngcntr}
\usepackage{apptools}
\AtAppendix{\counterwithin{Lem}{section}}

\def\R{\hbox{\font\dubl=msbm10 scaled 1100 {\dubl R}}}

\def\d{\,{\rm{d}}}
\def\Re{{\rm{Re}}\,}
\def\Im{{\rm{Im}}\,}
\newtheorem*{Theorem*}{Theorem}
\newtheorem{Theorem}{Theorem}

\newtheorem{Cor}{Corollary}

\newtheorem{Lem}{Lemma}

\title[]{{\large Riemann-Type Functional Equations}\\ {\footnotesize -- Julia Line and Counting Formulae --}}
\begin{document}
\author[Athanasios Sourmelidis, J\"orn Steuding, Ade Irma Suriajaya]{Athanasios Sourmelidis$^1$, J\"orn Steuding$^2$, Ade Irma Suriajaya$^3$}
\thanks{
	\hspace{-6.3mm}
	$^1$Institute of Analysis and Number Theory, TU Graz, Graz, Austria. \\
	$^2$Department of Mathematics, University of W\"urzburg, W\"urzburg, Germany.\\
	$^3$Faculty of Mathematics, Kyushu University, Fukuoka, Japan. \\
	{\bf Dedicated to Keith Matthews at the occasion of the 25th anniversary of his Number Theory Web.}}

\begin{abstract}
We study Riemann-type functional equations with respect to value-distribution theory and derive implications for their solutions. 
In particular, for a fixed complex number
$a\neq0$ and a function from the Selberg class $\mathcal{L}$, we prove a Riemann-von Mangoldt formula for the number of a-points of the
$\Delta$-factor of the functional equation of $\mathcal{L}$ and an analog of Landau's formula over these points. From the last formula we derive that the ordinates of these $a$-
points are uniformly distributed modulo one. 
Lastly, we show the existence of the mean-value of
the values of $\mathcal{L}(s)$ taken at these points.
\end{abstract}

\maketitle
{\small \noindent {\sc Keywords:} Functional Equation, Extended Selberg Class, $a$-points, Riemann-von Mangoldt Formula, Landau Formula}
\title{}
\\
{\bf Funding:} The first author was supported by Austrian Science Fund, Projects F-5512 and Y-901 and the third author is supported by JSPS KAKENHI Grant Number 18K13400.

\section{Introduction and Statement of the Main Results}

Riemann-type functional equations appear in the context of number-theoretical relevant Dirichlet series or Euler products, resp. their analytic or meromorphic continuation; the prototype is, of course, the identity 
\begin{align*}
\pi^{-s/2}\Gamma\left(\frac{s}{2}\right)\zeta(s)=\pi^{-(1-s)/2}\Gamma\left(\frac{1-s}{ 2}\right)\zeta(1-s).
\end{align*}
for the Riemann zeta-function $\zeta(s)$.
These functional equations play a crucial role in the distribution of values of the underlying zeta- and $L$-functions and, consequently, they have been investigated intensively, for example by Hans Hamburger \cite{hamburger} (proving the first converse theorem) and beyond that by Erich Hecke \cite{hecke} (establishing an important correspondence to modular forms), to name just two. Several attempts had been made to implement the ongoing studies into a bigger picture, e.g., Bruce Berndt \cite{berndt} studied the impact of such functional equations on the zero-distribution. The probably most far-reaching and furthermost accepted is the axiomatic setting suggested by Atle Selberg \cite{sel}. For our approach we have chosen this frame although alternatives would have been possible. 

The Selberg class $\mathcal{S}$ consists of Dirichlet series
\begin{align*}
\mathcal{L}(s):=\sum\limits_{n\geq1}\dfrac{f(n)}{n^s},\quad s:=\sigma+it,
\end{align*}
which are absolutely convergent in the half-plane $\sigma>1$, respectively their analytic continuations satisfying the following hypotheses:
\begin{enumerate}[I]

\item{\it Analytic continuation.} There exists an integer $k\geq0$ such that $(s-1)^k\mathcal{L}(s)$ is an entire function of finite order.

\item{\it Functional equation.} $\mathcal{L}(s)$ satisfies a functional equation of type 
\begin{align*}
\mathfrak{L}(s)=\omega\overline{\mathfrak{L}(1-\overline{s})},
\end{align*}
where
\begin{align*}
\mathfrak{L}(s):=\mathcal{L}(s)Q^s\prod\limits_{j\leq r}\Gamma\left(\lambda_js+\mu_j\right)
\end{align*}
with positive real numbers $Q$, $\lambda_j$, and complex numbers $\mu_j$, $\omega$ with $\Re\mu_j\geq0$ and $|\omega|=1$.

\item{\it Ramanujan hypothesis.} $f(n)\ll_\epsilon n^\epsilon$.

\item{\it Euler product.} $\mathcal{L}(s)$ has a product representation
\begin{align*}
\mathcal{L}(s)=\prod\limits_{p\text{ prime}}\mathcal{L}_p(s),
\end{align*} 
where 
\begin{align*}
\mathcal{L}_p(s)=\exp\left(\sum\limits_{k\geq1}\dfrac{b(p^k)}{p^{ks}}\right)
\end{align*}
with suitable coefficients $b(p^k)$ satisfying $b(p^k)\ll p^{k\theta}$ for some $\theta<1/2$.
\end{enumerate}

For our exposition we only need the first two axioms, which characterize the larger class of functions commonly known as the \emph{extended Selberg class}, denoted by $\mathcal{S}^\sharp$. In view of axiom II we can define the numbers
\begin{equation}\label{dlm}
d_\mathcal{L}:=2\sum\limits_{ j\leq r}\lambda_j,\quad\lambda:=\prod\limits_{ j\leq r}\lambda_j^{2\lambda_j}\quad\text{and}\quad\xi:=2\sum\limits_{j\leq r}\left(\mu_j-\dfrac{1}{2}\right)=:\eta+i\Theta.
\end{equation}
The quantities $d_\mathcal{L}$, $\lambda Q^2$ and $\xi$ are uniquely determined for a function $\mathcal{L}\in\mathcal{S}^\sharp$ and, hence, well-defined, even though the data from the functional equation is not unique (see \cite[Section 4]{kape1}). 
Moreover, we may rewrite the equation in axiom II in the following asymmetrical form
\begin{equation}\label{doublecross}
{\mathcal L}(s)=\Delta_{\mathcal L}(s)\overline{\mathcal L}(1-s),
\end{equation}
where $\overline{\mathcal{L}}(s):=\overline{\mathcal{L}(\overline{s})}$ and $\Delta_{\mathcal L}(s)$ is an exponential function times a product of quotients of Gamma-factors
\begin{equation}
\Delta_\mathcal{L}(s):=\omega Q^{1-2s}\prod\limits_{j\leq r}\dfrac{\Gamma\left(\lambda_j(1-s)+\overline{\mu_j}\right)}{\Gamma\left(\lambda_js+\mu_j\right)}.
\end{equation}
In order to have a Riemann-type functional equation we need $d_{\mathcal{L}}$ to be positive. Brian Conrey \& Amit Ghosh \cite{congo} showed that there are no elements in $\mathcal{S}^\sharp$ of degree $d\in(0,1)$. The degree one-elements have been classified by Jerzy Kaczorowski \& Alberto Perelli \cite{kape} (as linear combinations of Dirichlet $L$-functions with a potentially imaginary shift multiplied by certain Dirichlet polynomials). It is conjectured that the degree is always a non-negative integer. This has been verified by Kaczorowski \& Perelli \cite{kape2} up to degree $2$; and recently, they have also classified degree two-elements of conductor 1 \cite{kape3}. 

We want to stress here that these functions $\Delta_{\mathcal L}(s)$ arising from a Riemann-type functional equation are the main actors of this article; the corresponding solutions of the functional equation will first play only a subsidiary role. 
Given a complex number $a\neq0$, we define the solutions of
\begin{equation}
\Delta_\mathcal{L}(s)=a,\quad0\leq\sigma\leq 1,
\end{equation}
to be the \emph{nontrivial} $a$-points of $\Delta_\mathcal{L}(s)$ (in analogy to the nontrivial zeros of the Riemann zeta-function); in the sequel we shall denote them as $\delta_a=\beta_a+i\gamma_a$.

Our main results rely on the argument principle theorem in suitable rectangles $\mathcal{R}$, namely,
\begin{equation}\label{argprinc}
\sum_{0<\gamma_a\leq T}g(\delta_a)={\frac{1}{2\pi i}}\int_{\partial\mathcal{R}}\dfrac{\Delta_\mathcal{L}'(s)}{\Delta_\mathcal{L}(s)-a}g(s)\d s,
\end{equation}
where $g(s)$ is an analytic function in a neighbourhood of $\mathcal{R}$. The crucial point in these applications is the treatment of the upcoming error terms. Our first theorem is a Riemann-von Mangoldt-type formula.

\begin{Theorem}\label{RvMtheorem}
Let $a\neq0$ be a complex number, $\mathcal{L}\in\mathcal{S}^\sharp$ with $d_\mathcal{L}\geq1$ and $\psi(t):=\log t/\log\log t$, $t\geq3$.
If $N_\pm(T;a,\Delta_\mathcal{L})$ counts the number of nontrivial $a$-points $\delta_a=\beta_a+i\gamma_a$ of $\Delta_{\mathcal{L}}(s)$ satisfying $0<\pm\gamma_a\leq T$ (counting multiplicities), then
\begin{equation}\label{RvM}
N_\pm(T;a,\Delta_\mathcal{L})
=\dfrac{d_\mathcal{L}}{2\pi}T\log{T}+\dfrac{\log\left(\lambda Q^2\right)-d_\mathcal{L}}{2\pi}T\pm\dfrac{\Theta}{2\pi}\log T+O_{a,\mathcal{L}}\left(\psi(T)\right)
\end{equation}
for any $T>0$.
\end{Theorem}

Observe that the number of nontrivial $a$-points of $\Delta_\mathcal{L}(s)$ up to some height $T$ is asymptotically equal to the number of nontrivial zeros of the corresponding $\mathcal{L}(s)$, as follows from a well-known argument for the Riemann zeta-function $\zeta(s)$ (cf. \cite[\S 9.4]{tit}). This is not surpising because the main term of the counting function for the nontrivial zeros of $\mathcal{L}(s)$, that is
\begin{equation}
\dfrac{d_\mathcal{L}}{2\pi}T\log{T}+\dfrac{\log\left(\lambda Q^2\right)-d_\mathcal{L}}{2\pi}T,
\end{equation} 
results from integrating $\Delta'_\mathcal{L}(s)/\Delta_\mathcal{L}(s)$ on some suitable vertical line, while the only contribution of $\mathcal{L}(s)$ is included in error terms of order $O(\log T)$. As a matter of fact, it is the argument of $\mathcal{L}(1/2+iT)$ (obtained by continuous variation from some fixed point in the half-plane of absolute convergence of $\mathcal{L}(s)$) which yields the error term $O(\log T)$
, and improving upon that is one of the most challenging problems in analytic number theory. Here however, we are not dealing with a \emph{zeta-function} and obtaining a third main term of order $\log T$ and an error term of smaller order is possible in our case.

It is also worth mentioning that for $\Theta\neq0$ the nontrivial $a$-points of $\Delta_\mathcal{L}(s)$ are not symmetrically distributed with respect to the real line. This suggests that the nontrivial zeros of $\mathcal{L}(s)$ should also lie asymmetrically with respect to the real line, at least when $\mathcal{L}\in\mathcal{S}$ and satisfies the \emph{Grand Riemann Hypothesis}, that is, $\mathcal{L}(s)\neq0$ for $\sigma>1/2$. In this case we expect $\arg \mathcal{L}(1/2+it)=o(\log t)$, which is the case for $\zeta(s)$ under the Riemann hypothesis (see \cite[\S 14.13]{tit}). Hence, the third main term $\Theta(\log T)/2\pi$ which occurs from integrating $\Delta'_\mathcal{L}(s)/\Delta_\mathcal{L}(s)$ is not absorbed in the error term and the bias on the distribution of zeros makes its appearance. Incidentally, the number $\Theta$ is called the \emph{shift} of $\mathcal{L}(s)$ by Kaczorowski \& Perelli \cite{kape} for a different reason.

Obviously, $\Delta_{\mathcal L}(s)$ is a meromorphic function of order $1$ with two deficient values $0$ and $\infty$ in the sense of value-distribution or Nevanlinna theory. The latter statement is actually an immediate consequence of the counting formula for the number of $a$-points of $\Delta_{\mathcal L}(s)$ in Theorem \ref{RvMtheorem}. Nevanlinna theory has been considered as one of the two major generalizations of the celebrated Great Picard theorem claiming that an analytic function in a punctured neighbourhood of an essential singularity takes every complex value, with at most one exception, infinitely many times. This deep result was found by \'Emile Picard \cite{picard} in 1879. The other extension is due to Gaston Julia \cite{julia} who achieved in 1919 a remarkable refinement of Picard's theorem by showing that one can even add a restriction on the angle at the essential singularity to lie in an arbitrarily small cone (see also \cite[\S 12.4]{burckel}). In the context of Dirichlet series appearing in number theory, however, it is more natural to consider the so-called {\it Julia lines} rather than Julia directions; for this and further details we refer to \cite{steusur}. In view of 
$$
\left|\Delta_\mathcal{L}(\sigma+it)\right|=\left(\lambda Q^2|t|^{d_\mathcal{L}}\right)^{1/2-\sigma}\left(1+O\left(\dfrac{1}{|t|}\right)\right)\qquad\mbox{for}\quad \vert t\vert\geq 1
$$
(see Lemma \ref{StirlingsDelta}) and Theorem \ref{RvMtheorem} it follows that the critical line $1/2+i\R$ is a Julia line for $\Delta_{\mathcal L}(s)$. It is not dfifficult to see that the real axis is another Julia line for $\Delta_{\mathcal L}$ and that there is no other. This follows from a straightforward application of Rouch\'e's theorem along the lines for the Riemann zeta-function (see \cite{levi}).

For our second application of the argument principle \eqref{argprinc} we consider $g(s)=x^s$ with some positive number $x\neq1$. Similar investigations arise quite often in zeta-function theory. In the case of $\zeta(s)$, Edmund Landau \cite{land} proved that the explicit formula
\begin{equation}\label{Lan}
\sum\limits_{0<\gamma\leq T}x^\rho=-\dfrac{T}{2\pi}\Lambda(x)+O_x(\log T),
\end{equation}
holds for any $T>1$ and any positive number $x>1$, where $\rho=\beta+i\gamma$ denotes the nontrivial zeros of $\zeta(s)$ and $\Lambda(x)$ is the von Mangoldt $\Lambda$-function extended to the whole real line by setting it to be zero if $x$ is not a positive integer. 
The second author \cite{steu1} proved a similar result in the case of $a$-points of $\zeta(s)$, namely, for $a\neq1$,
\begin{equation}\label{St}\sum\limits_{0<\gamma_a\leq T}x^{\rho_a}=\dfrac{T}{2\pi}\left(a(x)-x\Lambda\left(\frac{1}{x}\right)\right)+O_{x,\epsilon}\left(T^{1/2+\epsilon}\right),
\end{equation}
where $\rho_a=\beta_a+i\gamma_a$ denotes the so called {\it nontrivial} $a$-points of $\zeta(s)$ and $a(x)$ is a certain computable arithmetical function extended to the whole real line. 
The dependence of the implicit constants on $x$ in \eqref{Lan} and \eqref{St} had been made explicit by Steven M. Gonek \cite{gonek2}, Martin Rehberg \cite{reh} and Siegfried Baluyot \& Gonek \cite{balgon}.
Compared to these results, the asymptotics given in the following theorem for the case of nontrivial $a$-points of $\Delta_\mathcal{L}(s)$ are rather different. 

\begin{Theorem}\label{Landau}
Let $a\neq0$ be a complex number, $\mathcal{L}\in\mathcal{S}^\sharp$ with $d_\mathcal{L}\geq1$ and $y(t)$ be a function such that $2\leq y(t)\leq{T}/|4\Theta|$. Then there exists $T_0=T_0(a,\mathcal{L},y)>0$ such that
for any $x\neq1$ and any $T,T'$ with $T_0\leq T<T+1/\log T\leq T'\leq2T$ satisfying
\begin{equation}\label{length}
0<\max\left\{ x,\dfrac{1}{x}\right\}\leq \left(1-\dfrac{1}{y(T)}\right)\lambda Q^2T^{d_\mathcal{L}},
\end{equation} 
we have
\[
\sum_{\substack{T<\gamma_a\leq T'}} x^{\delta_a}
=\dfrac{x^{1/2}M(x,T,T')}{2\pi i|\log x|}+O_{a,\mathcal{L},y}\left(x^{1/2}\left(\psi(T)+y(T)|\log x|+\dfrac{1}{|\log x|}\right)\right),
\]
where $\delta_a$ are the nontrivial $a$-points of $\Delta_\mathcal{L}(s)$ and
\begin{align}\label{MainTerm}
M\left(x,T,T'\right):=x^{iT'}\log\left(\lambda Q^2\left(T'\right)^{d_\mathcal{L}}\right)-x^{iT}\log\left(\lambda Q^2T^{d_\mathcal{L}}\right).
\end{align}
\end{Theorem}

\noindent The main term in Theorem \ref{Landau} ranges between $x^{1/2}$ and $x^{1/2}\log T$. Therefore, for certain ranges of $x$ and $T$ our theorem yields an asymptotic formula.

In the succeeding sections we will show that, for fixed complex number $a\neq0$, the nontrivial $a$-points of $\Delta_\mathcal{L}(s)$ cluster in a particular manner near the vertical line $1/2+i\mathbb{R}$.
In combination with Theorem \ref{Landau} we will also be able to obtain a result regarding the vertical distribution of those points.
\begin{Cor}\label{UDmod}
Let $a$, $\mathcal{L}(s)$, $y(t)$, $T$ and $x$ be defined as in Theorem \ref{Landau}. 
If $(\gamma_a)_{\gamma_a>0}$ is the increasing sequence of imaginary parts of the nontrivial $a$-points $\delta_a=\beta_a+i\gamma_a$ of $\Delta_\mathcal{L}(s)$, then
\begin{equation}\label{imagin}
\sum_{T<\gamma_a\leq2T}x^{i\gamma_a}\ll_{a,\mathcal{L},y}(1+|\log x|)\left(\dfrac{\log T}{|\log x|}+\psi(T)+y(T)|\log x|\right).
\end{equation}
In particular, the sequence $(\alpha\gamma_a)_{\gamma_a>0}$ is uniformly distributed modulo one in $\mathbb{R}$ for every real number $\alpha\neq0$.
\end{Cor}
Recall here that a sequence of real numbers $(x_n)_{n\in\mathbb{N}}$ is called {\it uniformly distributed modulo one} if
\[
\lim\limits_{N\to\infty}\dfrac{1}{N}\sharp\left\{1\leq n\leq N:\lbrace x_n\rbrace\in[a,b]\right\}=b-a
\]
for any real numbers $0\leq a\leq b\leq1$, where $\lbrace x\rbrace:=x-\lfloor x\rfloor$ for $x\geq0$. 

For the final application of the argument principle \eqref{argprinc} we consider $g=\mathcal{L}$. This problem is motivated by an earlier work of the second author in collaboration with Justas Kalpokas \cite{kapsteu} and the third author \cite{steusur}.

\begin{Theorem}\label{moment}
Let $a\neq0$ be a complex number, $\mathcal{L}\in\mathcal{S}^\sharp$ with $d_\mathcal{L}\geq1$ and $A\ge0$ be such that $\mathcal{L}(1/2+it)\ll_\epsilon|t|^{A+\epsilon}$ as $t\to\infty$.
Then, for every $T>0$,
\begin{align}\label{firstmoments}
\sum_{0<\gamma_a\leq T}{\mathcal L}(\delta_a)
=(f(1)+a\overline{f(1)})N_+(T;a,\Delta_\mathcal{L})+ O_{a,\mathcal{L},\epsilon}\left(T^{A+\epsilon}\right).
\end{align}
\end{Theorem}
 
\noindent Observe that, in general, a Riemann-type functional equation has several solutions in the (extended) Selberg class; for example, all Dirichlet $L$-functions $L(s;\chi)$ associated with odd characters $\chi\bmod\,5$ satisfy the same functional equation.

The error term in \eqref{firstmoments} results from the estimate of the integral \eqref{argprinc} on its horizontal sides.
We describe this in a little more detail. Using a partial fractional decomposition of $\Delta'_\mathcal{L}(s)/(\Delta_\mathcal{L}(s)-a)$ on these horizontal paths allows us to bound it by some power of $\log{T}$. Meanwhile it suffices to bound $\mathcal{L}(1/2+iT)$ trivially by $T^{A+\epsilon}$.
In general, one has $\mathcal{L}(1/2+it)\ll_\epsilon t^{d_\mathcal{L}/4+\epsilon}$ as $t\to\infty$, which is known as the \emph{convexity} bound (see \cite[Theorem 6.8]{steudi} or \eqref{esti}).
Hence, for elements $\mathcal{L}\in\mathcal{S}^\sharp$ with $d_\mathcal{L}<4$, Theorem \ref{moment} yields an asymptotic formula, while for elements of $\mathcal{S}^\sharp$ with degree $d_\mathcal{L}\geq 4$ ``sharper" estimates of $\mathcal{L}(1/2+it)$ are needed, namely $\mathcal{L}(1/2+it)\ll_\epsilon t^{A+\epsilon}$ for some positive number $A<1$.
However, such \emph{subconvexity} estimates usually amount to understanding the arithmetical behavior of the coefficients $f(n)$ of the Dirichlet series expansion of $\mathcal{L}(s)$. The most simple example is the Riemann zeta-function for which $\zeta(1/2+it)\ll_\epsilon t^{13/84+\epsilon}$ holds, recently proven by Jean Bourgain \cite{bourg}. Applying this bound to \eqref{firstmoments}, it is obvious that in the case of $\zeta^{k}(s)$, $k\leq6$, the main term in the estimate \eqref{firstmoments} is dominant. On the other hand, we have not assumed anything on the arithmetical nature of the coefficients $f(n)$ and thus $\mathcal{L}(1/2+it)\ll_\epsilon t^{d_\mathcal{L}/4+\epsilon}$ is currently the best known bound for elements of $\mathcal{S}^\sharp$ of large degree. We remark that if the \emph{Lindel\"of Hypothesis} for $\mathcal{L}$ is true, namely
\begin{equation}\label{GLH}
\mathcal{L}(1/2+it)\ll_\epsilon t^\epsilon\qquad\mbox{as}\quad t\to\infty,
\end{equation}
then Theorem \ref{moment} yields an asymptotic formula for arbitrary $\mathcal{L}\in\mathcal{S}^\sharp$.

\subsection*{Structure of the paper and notations} 

All constants appearing in the sequel depend on the complex number $a$ and the function $\mathcal{L}$. In particular, the constants depend also on $\lambda_j,\Im\mu_j$ and $Q$. Since our results are with respect to fixed $a$ and $\mathcal{L}$, we will drop any such subscripts in the Landau symbols $O(\cdot)$, $o(\cdot)$ or the Vinogradov symbols $\ll$, $\gg$. Moreover, $f\asymp g$ simply means that $|f|\ll |g|\ll |f|$.
At the end of this paper, we also provide an appendix regarding properties of $\Delta_\mathcal{L}(s)$, which mainly follow from an application of Stirling's formula, and which will be utilized frequently in the sequel.



\section{A Riemann-von Mangoldt Formula --- Proof of Theorem \ref{RvMtheorem}}\label{SectRvMtheorem}

Let $a\neq0$ be a complex number with $|a|\leq1$. Assume also that $\pm T$ are not ordinates of some $a$-point.
For simplicity we define here
\begin{equation}\label{abbreviation}
\ell(t)=\log\left({\lambda Q^2|t|^{d_\mathcal{L}}}\right).
\end{equation}
for every $|t|>0$.
Then Stirling's formula for $\Delta_\mathcal{L}$, (\ref{Deltasymp}), yields
\begin{equation} \label{delta_asymp}
\Delta_\mathcal{L}(\sigma+it) \asymp e^{(1/2-\sigma)\ell(t)},\quad|t|\geq t_0>0.
\end{equation}
Moreover, since $|\Delta_\mathcal{L}(\beta_a+i\gamma_a)|=|a|$, \eqref{delta_asymp} implies that
\begin{equation}\label{rp}
\beta_a=\dfrac{1}{2}-\dfrac{\log|a|}{\ell({\gamma}_a)}+O\left(\dfrac{1}{\gamma_a\ell({\gamma}_a)}\right)\in(0,1).
\end{equation}
for any $|\gamma_a|\geq T_0$, where $T_0>0$ is assumed to be sufficiently large.
Let $|t|\geq T_0$ and
\begin{equation}\label{alpha}
\alpha(t) := \dfrac12+\dfrac{c}{\ell({t})}
\end{equation}
where $c=c(a,\mathcal{L})>0$ is a sufficiently large constant satisfying
\begin{equation}\label{dist}
\left|\beta_a-\dfrac{1}{2}\right|\leq\dfrac{c}{2\ell({t})}
\end{equation}
for any $a$-point with $|\gamma_a|\geq |t|\geq T_0$ as follows from \eqref{rp}. Moreover,
\begin{equation}\label{inequality1}
\begin{array}{cccccccccc}
|\Delta_\mathcal{L}\left(\alpha(u)+it\right)|&\leq& Ke^{-c\ell\left({t}\right)/\ell\left({u}\right)}&=&Ke^{-c}&<&\dfrac{|a|}{2},
\end{array}
\end{equation}
\begin{equation}\label{inequality2}
\begin{array}{cccccccccc}
|\Delta_\mathcal{L}\left(1-\alpha(u)+it\right)|&=&|\Delta_\mathcal{L}\left(\alpha(u)+it\right)|^{-1}&\geq& \dfrac{e^{c}}{K}&\geq&2|a|,
\end{array}
\end{equation}
for any $|t|\geq|u|\geq T_0$, where $K$ is a constant occuring from \eqref{delta_asymp} and we can assume without loss of generality that $K>1$. After choosing a suitable $K$, we then take $c$ sufficiently large. Lastly, Lemma \ref{Value-DistributionDelta} implies that there exists $T_1\geq T_0$, which we can assume without loss of generality not to be an ordinate of some $a$-point, such that if $\Delta_\mathcal{L}(s)=a$ for some $|t|\geq T_1$, then $\sigma\geq1/2$.

With the above settings, we find that \eqref{alpha}, \eqref{dist} and the calculus of the residues imply
\begin{eqnarray*}
2\pi N_\pm(T;a,\Delta_\mathcal{L})
&=&2\pi\left(N_\pm(T;a,\Delta_\mathcal{L})-N_\pm(T_1;a,\Delta_\mathcal{L})\right)+O(1)\\
&=&V_{\mathcal{C}_\pm}\arg\left(\Delta_{\mathcal{L}}(s)-a\right)+O(1),
\end{eqnarray*}
where $V_{\mathcal{C_\pm}}$ denotes the variation around the counterclockwise oriented contour 
$$\mathcal{C}_\pm:=\bigcup_{i\leq4} L_i^\pm$$
 and the paths $L_i^\pm$, $ 1\leq i\leq 4$, are defined as follows:
\begin{equation*}
\begin{array}{ll}
L_1^+=\left[\alpha(T_1)+iT_1,\alpha(T)+iT\right],&L_1^-=\left[\alpha(T)-iT,\alpha(T_1)-iT_1\right],\\
L_2^+=\left[\alpha(T)+iT,1-\alpha(T)+iT\right],&L_2^-=\left[\alpha(T_1)-iT_1,1-\alpha(T)-iT_1\right],\\
L_3^+=\left[1-\alpha(T)+iT,1-\alpha(T)+iT_1\right],&L_3^-=[1-\alpha(T)-iT_1,1-\alpha(T)-iT],
\\
L_4^+=\left[1-\alpha(T)+iT_1,\alpha(T_1)+iT_1\right],&L_4^-=\left[1-\alpha(T)-iT,\alpha(T)-iT\right].
\end{array}
\end{equation*}

Since $L_4^+$ and $L_2^-$ are subsets of the line segments $[1-\alpha(T_1)\pm iT_1,\alpha(T_1)\pm iT_1]$, respectively, and therefore independent of $T\geq T_1$, we obtain
\begin{equation}\label{finite}
V_{L_4^+}\arg\left(\Delta_\mathcal{L}(s)-a\right)=O(1)=V_{L_2^-}\arg\left(\Delta_\mathcal{L}(s)-a\right).
\end{equation}
In addition, inequalities \eqref{inequality1} imply
\begin{equation}
V_{L_1^\pm}\arg\left(\Delta_\mathcal{L}(s)-a\right)=V_{L_1^\pm}\arg(-a)+V_{L_1^\pm}\arg\left(1-\dfrac{\Delta_\mathcal{L}(s)}{a}\right)=0+O(1).
\end{equation}
In view of \eqref{inequality2} and Lemma \ref{StirlingsDelta}, it follows that
\begin{eqnarray}
V_{L_3^\pm}\arg\left(\Delta_{\mathcal{L}}(s)-a\right)\nonumber
&=&V_{L_3^\pm}\arg\left(\Delta_{\mathcal{L}}(s)\right)+V_{{L_3^\pm}}\arg\left(1-\dfrac{a}{\Delta_{\mathcal{L}}(s)}\right)\\
&=&\Im\int_{L_3^\pm}\dfrac{\Delta'_\mathcal{L}}{\Delta_\mathcal{L}}(s)\mathrm{d}s+O(1)\nonumber\\
&=&\int_{T_1}^{T}\left(\log\left(\lambda Q^2t^{d_\mathcal{L}}\right)\pm\dfrac{\Theta}{t}+O\left(\dfrac{1}{t^2}\right)\right)\mathrm{d}t+O(1)\nonumber\\
&=&Td_\mathcal{L}\log{T}+T\log\left(\dfrac{\lambda Q^2}{e^{d_\mathcal{L}}}\right)\pm\Theta\log T+O(1).
\end{eqnarray}

Next we prove
$$V_{L_2^+}\arg\left(\Delta_{\mathcal{L}}(s)-a\right)\ll\dfrac{\log T}{\log\log T}$$
in a similar manner to how $\arg\zeta(1/2+iT)$ is estimated (see for example \cite[\S 9.4]{titchm}).
Suppose that $\Re\left(\Delta_\mathcal{L}(\sigma+iT)-a\right)$ has $N$ zeros for $1-\alpha(T)\leq\sigma\leq\alpha(T)$.
We divide the interval $[1-\alpha(T),\alpha(T)]$ into at most $N+1$ subintervals in each of which $\Re\left(\Delta_\mathcal{L}(\sigma+iT)-a\right)$ is of constant sign.
Thus,
\begin{equation}\label{VL4}
\left|V_{L_2^+}\arg\left(\Delta_{\mathcal{L}}(s)-a\right)\right|\leq(N+1)\pi.
\end{equation}
To estimate $N$ let 
$$g(z):=\dfrac{1}{2}\left(\Delta_\mathcal{L}(z+iT)-a+\overline{\Delta_\mathcal{L}(\overline{z}+iT)}-\overline{a}\right).$$
Then $g(\sigma)=\Re\left(\Delta_\mathcal{L}(\sigma+iT)-a\right)$, and thus $N$ is equal to the number of zeros of $g(z)$ for $\Im z=0$ and $1-\alpha(T)\leq\Re z\leq\alpha(T)$.
Hence, if $n(r)$ denotes the number of zeros of $g(z)$ inside the disk $|z-\alpha(T)|\leq r$, then 
\begin{equation}
N\leq n(2\alpha(T)-1).
\end{equation}
On the other hand, we have
\begin{equation}
\int_{0}^{\alpha(T)}\dfrac{n(r)}{r}\mathrm{d}r\geq\int_{2\alpha(T)-1}^{\alpha(T)}\dfrac{n(r)}{r}\mathrm{d}r\geq n(2\alpha(T)-1)\log\dfrac{\alpha(T)}{2\alpha(T)-1},
\end{equation}
where the inequality $\alpha(T)\geq2\alpha(T)-1$ holds for any $T\geq T_0$ and sufficiently large $T_0>0$.
Since $g(z)$ is analytic in the disk $|z-\alpha(T)|\leq \alpha(T)+1$ and $g(\alpha(T))\neq0$ from \eqref{inequality1}, Jensen's theorem \cite[\S 3.61]{titch} yields
\begin{equation}\label{Jensen's}
\int_{0}^{\alpha(T)}\dfrac{n(r)}{r}\mathrm{d}r=\dfrac{1}{2\pi}\int_0^{2\pi}\log\left|g\left(\alpha(T)+\alpha(T) e^{i\theta}\right)\right|\mathrm{d}\theta-\log\left|g(\alpha(T))\right|.
\end{equation}
Observe that from \eqref{delta_asymp} we know that the right integral above is less than $A\log T$ for some constant $A>0$.
Moreover, \eqref{inequality1} implies $|g(\alpha(T))|\asymp1$.
Therefore the integral on the left-hand side of \eqref{Jensen's} is less than $A\log T$ for some constant $A>0$.
It follows now from \eqref{VL4}--\eqref{Jensen's} and the definitions \eqref{alpha} and \eqref{abbreviation} of $\alpha(T)$ and $\ell(T)$, respectively, that
$$V_{L_2^+}\arg\left(\Delta_{\mathcal{L}}(s)-a\right)
\ll\log T\left(\log\dfrac{\frac{1}{2}+\frac{c}{\ell(T)}}{\frac{2c}{\ell(T)}}\right)^{-1}
\ll\dfrac{\log T}{\log\log T}.$$
Similarly, we can show that
$$V_{L_4^-}\arg\left(\Delta_{\mathcal{L}}(s)-a\right)\ll\dfrac{\log T}{\log\log T}.$$
This concludes the proof of the theorem when $|a|\leq1$ and $\pm T$ are not ordinates of some $a$-point.

If $T$ or $-T$ is an ordinate of some $a$-point, then we apply the residue theorem in contours $\mathcal{C}_\pm$, described as before, but we integrate over horizontal segments of height $T+\varepsilon$ or $-T-\varepsilon$, respectively, for any sufficiently small $\varepsilon>0$, such that the intervals $(T,T+\varepsilon]$, $[-T-\varepsilon,-T)$ do not contain any ordinates of $a$-points. One can see that the proof is independent of our choice of $\varepsilon$. Therefore, taking $\varepsilon$ to $0$ yields \eqref{RvM}. The case of complex numbers $a\neq0$ with $|a|>1$ will follow from the identity
\begin{equation}\label{SymmetryIdentity1}
\Delta_\mathcal{L}(s)\overline{\Delta_\mathcal{L}(1-\overline{s})}=1,
\end{equation}
which implies that a complex number $z$ is an $a$-point of $\Delta_\mathcal{L}(s)$ if, and only if, $1-\overline{z}$ is an $1/\overline{a}$-point of $\Delta_\mathcal{L}(s)$.


\section{A Landau Formula --- Proof of Theorem \ref{Landau}}

Let $x>1$. 
We keep the notations \eqref{abbreviation}--\eqref{inequality2} in the proof of Theorem \ref{RvMtheorem}, which implies that if $T_0\gg1$ is a sufficiently large number, then for any real numbers $T_0\leq T<T+1/\log T\leq T'\leq2T$, we can find 
\begin{equation}
T_1\in\left[T,T+1/(2\log T)\right)\quad\text{ and }\quad T_2\in\left(T'-1/(2\log T),T'\right]
\end{equation}
 such that
\begin{align}\label{T_i}
\min\limits_{k=1,2}\min\limits_{\delta_a}\left|T_k-\gamma_a\right|\gg\dfrac{\log\log T}{\log T}.
\end{align}

If $\mathcal{C}$ is the positively oriented rectangular contour with vertices $\alpha(T)+iT_1$, $\alpha(T)+iT_2$, $1-\alpha(T)+iT_2$, $1-\alpha(T)+iT_1$, then Theorem \ref{RvMtheorem}, inequality \eqref{dist} and the calculus of residues yield
\begin{eqnarray}\label{anf}
 \sum_{\substack{T<\gamma_a\leq T'}} x^{\delta_a}&=& \sum_{\substack{T_1<\gamma_a<T_2}} x^{\delta_a}+\sum_{\substack{T<\gamma_a<T_1}}x^{\delta_a}+\sum_{\substack{T_2<\gamma_a\leq T'}} x^{\delta_a}\nonumber\\
 &=&\frac{1}{ 2\pi i} \int_{\mathcal{C}} \frac{\Delta'(s)}{\Delta(s)-a} x^s\, \d s+O\left(x^\alpha\psi(T)\right).
\end{eqnarray}

Let $\alpha:=\alpha(T)$. 
We break the integral
$$ \frac{1}{ 2\pi i} \int_{\mathcal{C}} \frac{\Delta'(s)}{\Delta(s)-a} x^s\, \d s $$
down into
$$
\frac{1}{2\pi i} \left\{\int_{\alpha+iT_1}^{\alpha+iT_2} + \int_{\alpha+iT_2}^{1-\alpha+iT_2} + \int_{1-\alpha+iT_2}^{1-\alpha+iT_1} + \int_{1-\alpha+iT_1}^{\alpha+iT_1} \right\} \frac{\Delta'(s)}{\Delta(s)-a} x^s \,\d{s} =: \sum_{j\leq 4}I_j.
$$

In view of Lemma \ref{fractionaldecomposition}, for $k= 1, 2$ we have
\begin{eqnarray*}
\int_{1-\alpha+iT_k}^{\alpha+iT_k} \frac{\Delta'(s)}{\Delta(s)-a} x^s \,\d{s}
&=& \int_{1-\alpha+iT_k}^{\alpha+iT_k} \left(\sum_{\vert t-\gamma_a\vert\leq 1/\log\log t}\frac1{s-\delta_a} + O(\log{t})\right) x^s \,\d{s} \\
&=& \int_{1-\alpha+iT_k}^{\alpha+iT_k} \sum_{\vert t-\gamma_a\vert\leq 1/\log\log t} \frac1{s-\delta_a} x^s \,\d{s}
+ O\left(x^\alpha\right),
\end{eqnarray*}
since
$$
\int_{1-\alpha+iT_k}^{\alpha+iT_k} O(\log{t}) x^s \,\d{s}
\ll \log T_k \int_{1-\alpha}^\alpha x^\sigma\, \d\sigma
\ll x^\alpha (2\alpha-1) \log T \ll x^\alpha.
$$
Meanwhile, inequality \eqref{T_i} gives
\begin{eqnarray*}
\lefteqn{\int_{1-\alpha+iT_k}^{\alpha+iT_k} \sum_{\vert t-\gamma_a\vert\leq 1/\log\log t} \frac1{s-\delta_a} x^s \,\d{s}}\\
&\ll& \sum_{\vert T_k-\gamma_a\vert\leq 1/\log\log T_k} \int_{1-\alpha}^{\alpha} \frac{x^\sigma}{\left|\alpha-\beta_a+i(T_k-\gamma_a)\right|} \,\d{\sigma} \\
&\ll& x^\alpha (2\alpha-1) \dfrac{\log T}{\log\log T} \sum_{\vert T_k-\gamma_a\vert\leq 1/\log\log T_k} 1 \\
&\ll& x^\alpha\dfrac{\log T}{\left(\log\log T\right)^2}.
\end{eqnarray*}
Therefore,
\begin{equation} \label{hor}
I_2,I_4 \ll x^\alpha\psi(T).
\end{equation}

We now estimate the vertical integrals $I_1$ and $I_3$. 
For $\sigma=\alpha$, we can write
\begin{equation} \label{delta_i1_rewrite}
\frac{\Delta_\mathcal{L}'(s)}{\Delta_\mathcal{L}(s)-a} = \frac{\Delta_\mathcal{L}'(s)}{ -a} \cdot \frac{1}{1 - \Delta_\mathcal{L}(s)/a}
= \frac{\Delta_\mathcal{L}'(s)}{ -a} \left( 1 + \sum_{j\geq1} \left(\frac{\Delta_\mathcal{L}(s)}{a}\right)^j \right);
\end{equation}
here inequality \eqref{inequality1} allows us to expand the second factor into a geometric series. Applying this, we find that
\begin{equation}\label{I_1ref}
I_1 = -\frac{x^{\alpha}}{2\pi a} \int_{T_1}^{T_2} \Delta_\mathcal{L}'(\alpha+it) \left\{ \sum_{0\leq j< m} + \sum_{j\geq m}\right\} \left(\frac{\Delta_\mathcal{L}(\alpha+it)}{a}\right)^j x^{it} \,\d{t}.
\end{equation}
We choose now $m$, depending on $T$, large enough to bound the last term in the integral above trivially. 
Indeed, \eqref{inequality1} implies
\begin{equation} \label{delta_i1_asymp}
\sum_{j\ge m} \left(\frac{\Delta_\mathcal{L}(\alpha+it)}{a}\right)^j
= \frac{ \left(\frac{\Delta_\mathcal{L}(\alpha+it)}{a} \right)^m}{1-\frac{\Delta_\mathcal{L}(\alpha+it)}{a}}
\ll\dfrac{1}{2^m}.
\end{equation}
Hence, if we take
\begin{equation*}
m := \left\lfloor\dfrac{2\ell(T)}{\log 2}\right\rfloor,
\end{equation*}
then it follows from \eqref{delta_i1_asymp} that
\begin{eqnarray*}
I_1 &=& -\frac{x^{\alpha}}{2\pi a} \int_{T_1}^{T_2} \Delta_\mathcal{L}'(\alpha+it) \sum_{0\leq j< m} \left(\frac{\Delta_\mathcal{L}(\alpha+it)}{a}\right)^jx^{it} \,\d{t} \\
&&+~ O\left(x^\alpha\int_{T_1}^{T_2}\left|\Delta_\mathcal{L}'(\alpha+it)\right|\left(\dfrac{1}{\lambda Q^2T^{d_\mathcal{L}}}\right)^{2} \d{t}\right).
\end{eqnarray*}
The first term of the integrand in the error term can be estimated using Lemma \ref{StirlingsDelta} and \eqref{inequality1} for which we obtain
\begin{eqnarray*}
\Delta_\mathcal{L}'(\alpha+it)=\dfrac{\Delta_\mathcal{L}'(\alpha+it)}{\Delta_\mathcal{L}(\alpha+it)}\Delta_\mathcal{L}(\alpha+it)\ll\log\left(\lambda Q^2T^{d_\mathcal{L}}\right).
\end{eqnarray*}
It then follows that we can discard the last term in $I_1$ as
$$
I_1 = -\frac{x^{\alpha}}{2\pi a} \int_{T_1}^{T_2} \Delta_\mathcal{L}'(\alpha+it) \sum_{0\leq j< m} \left(\frac{\Delta_\mathcal{L}(\alpha+it)}{a}\right)^j x^{it} \,\d{t} + O\left(x^{\alpha}\right).
$$
Since
$$
\frac\d{\d{t}} \Delta_\mathcal{L}(\alpha+it)^{j+1} = i(j+1)\Delta_\mathcal{L}'(\alpha+it)\Delta_\mathcal{L}(\alpha+it)^j,
$$
we can rewrite $I_1$ as
\begin{eqnarray}\label{I_1}
I_1 &=& -\frac{x^{\alpha}}{2\pi i}\sum_{1\leq j\le m} \dfrac{1}{ja^j} \int_{T_1}^{T_2} \left(\Delta_\mathcal{L}(\alpha+it)^{j}\right)' x^{it} \,\d{t} + O\left(x^{\alpha}\right)\nonumber\\
&=:& -\frac{x^{\alpha}}{2\pi i}\sum_{j\le m} \dfrac{1}{ja^j} I_{1j}+ O\left(x^{\alpha}\right).
\end{eqnarray}
We now estimate the integrals $I_{1j}$, $j\le m$. Recall that $\alpha=\alpha(T)$.
Integrating by parts, we obtain with the aid of \eqref{inequality1} and Lemma \ref{StirlingsDelta} that
\begin{eqnarray*}
I_{1j}&=&\Delta_\mathcal{L}(\alpha+it)^{j} x^{it}\Big|_{T_1}^{T_2}-i\log x\int_{T_1}^{T_2}\Delta_\mathcal{L}(\alpha+it)^{j}x^{it}\d{t}\nonumber\\
&\ll&\left(\dfrac{|a|}{2}\right)^j+\log x\left|\int_{T_1}^{T_2}(\omega^*)^j\left(\lambda Q^2t^{d_\mathcal{L}}\right)^{(1/2-\alpha-it)j}\times\right.\nonumber\\
&&\hspace*{3.8cm}\left.\times\exp\left(ij\left(d_\mathcal{L}t-\Theta\log t\right)\right)\left(1+O\left(t^{-1}\right)\right)^jx^{it}\d{t}\right|
\end{eqnarray*}
The binomial identity implies that if $|O\left(t^{-1}\right)|<D/t$ for some $D>1$, then
\[ 
\left(1+O\left(t^{-1}\right)\right)^j=1+O\left((2D)^jt^{-1}\right). 
\]
In view of \eqref{inequality1}, we then have for a sufficiently large constant $K$ that
\begin{eqnarray}\label{I_1j}
I_{1j} &\ll&\log x\left|\int_{T_1}^{T_2}\left(\lambda Q^2t^{d_\mathcal{L}}\right)^{-jc/\ell(T)-ijt}\exp\left(ij\left(d_\mathcal{L}t-\Theta\log t\right)\right)x^{it}\mathrm{d}t\right| \nonumber\\
&&+~ \log x\int_{T_1}^{T_2}\left(2D\left(\lambda Q^2t^{d_\mathcal{L}}\right)^{-c/\ell(T)}\right)^jt^{-1}\d{t}+ \left(\dfrac{|a|}{2}\right)^j\nonumber\\
&\ll& \left|J\right|\log x+ \left(2D\dfrac{|a|}{2K}\right)^j \log x \int_{T_1}^{T_2}t^{-1}\d{t}+\left(\dfrac{|a|}{2}\right)^j\nonumber\\
&\ll& \left|J\right|\log x+ (1+\log x)\left(\dfrac{|a|}{2}\right)^j,
\end{eqnarray}
where
\begin{eqnarray*}
\mathcal{J}
&:=&\int_{T_1}^{T_2}\left(\lambda Q^2t^{d_\mathcal{L}}\right)^{-jc/\ell(T)}\exp\left[ij\left(t\ell(t)-d_\mathcal{L}t+\Theta\log t-\dfrac{t}{j}\log x \right)\right]\d{t}\\
&=&\int_{T_1}^{T_2}\left(\lambda Q^2t^{d_\mathcal{L}}\right)^{-jc/\ell(T)}\exp\left[ij\left(d_\mathcal{L}t+\Theta\right)\log\left(\left(\dfrac{\lambda Q^2}{x^{1/j}}\right)^{1/d_\mathcal{L}}\dfrac{t}{e}\right) \right]\d{t}\times\\
&&\times\exp\left[-ij\Theta\log\left(\left(\dfrac{\lambda Q^2}{x^{1/j}}\right)^{1/d_\mathcal{L}}\dfrac{1}{e}\right)\right].
\end{eqnarray*}

In order to estimate $J$ we employ the {\it first derivative test} (see for example \cite[Lemma 2.1]{ivic}):

{\it Let $F(x)$ be a real-valued differentiable fuction in an interval $[a,b]$ such that $F'(x)$ is monotonic and $|F'(x)|\geq F^{-1}>0$ throughout $[a,b]$. Let also $G(x)$ be positive and monotonic with $G(x)\ll G$ throughout $[a,b]$
	Then
	\begin{align}\label{firstderivativetest}
	\int_{a}^bG(x)\exp(iF(x))\d x\ll GF.
	\end{align}}
In our case we have that
\begin{align}\label{firstdertest}
\frac{\d}{\d t}\left[j\left(d_\mathcal{L}t+\Theta\right)\log\left(\left(\dfrac{\lambda Q^2}{x^{1/j}}\right)^{1/d_\mathcal{L}}\dfrac{t}{e}\right)\right]=j\log\left(\dfrac{\lambda Q^2t^d_\mathcal{L}}{x^{1/j}}\right)+\frac{j\Theta}{t}.
\end{align}
However, we assumed that 
$$1<x\leq\left(1-\dfrac{1}{y(T)}\right){\lambda Q^2T^{d_\mathcal{L}}},$$
and, thus,
\[
\log\left(\dfrac{\lambda Q^2t^{d_\mathcal{L}}}{x^{1/j}}\right)\geq-\log\left(1-\frac{1}{y(T)}\right)\geq\frac{1}{2y(T)}.
\]
Our assumption on $y(T)$, relation \eqref{firstdertest} and the first derivative test imply that
\begin{eqnarray}\label{Jay}
\mathcal{J}
\ll\left(\lambda Q^2T^{d_\mathcal{L}}\right)^{-jc/\ell(T)}y(T)
\ll\left(\dfrac{|a|}{2K}\right)^jy(T),
\end{eqnarray}
where the last relation follows from \eqref{inequality1}.

Now in view of \eqref{I_1}--\eqref{Jay}, we have
\begin{eqnarray}\label{sum1}
I_1&\ll& 
x^\alpha\sum\limits_{j\le m}\dfrac{1}{j|a|^j}\left(\left(\dfrac{|a|}{2K}\right)^jy(T)\log x+(1+\log x)\left(\dfrac{|a|}{2}\right)^j\right)+x^\alpha\nonumber\\
&\ll&x^\alpha(1+y(T)\log x).
\end{eqnarray}

For $\sigma=1-\alpha$, we can write
\begin{equation} \label{delta_i3_rewrite}
\frac{\Delta_\mathcal{L}'(s)}{\Delta_\mathcal{L}(s)-a} 
= \frac{\Delta_\mathcal{L}'}{\Delta_\mathcal{L}}(s) \cdot \frac{1}{1 - a/\Delta_\mathcal{L}(s)}
=\frac {\Delta_\mathcal{L}'}{\Delta_\mathcal{L}}(s) \left( 1 + \sum_{j\ge1} \left(\frac{a}{\Delta_\mathcal{L}(s)}\right)^j \right);
\end{equation}
here \eqref{inequality2} allows us to expand the second factor into a geometric series.
Thus, the left vertical integral $I_3$ can be decomposed as follows
\begin{eqnarray*}
I_3 &= &\frac{1 }{ 2\pi i} \int_{1-\alpha+iT_2}^{1-\alpha+iT_1} \frac{\Delta_\mathcal{L}'(s)}{\Delta_\mathcal{L}(s)-a} x^s \,\d{s} \\
&= &\dfrac{1}{2\pi i}\int_{1-\alpha+iT_2}^{1-\alpha+iT_1}\dfrac{\Delta_\mathcal{L}'}{\Delta_\mathcal{L}}(s)x^s \d{s}
+ \dfrac{1}{2\pi i}\int_{1-\alpha+iT_2}^{1-\alpha+iT_1}\dfrac{\Delta'_\mathcal{L}}{\Delta_\mathcal{L}}(s)x^s\sum_{j\ge1} \left(\frac{a}{\Delta_\mathcal{L}(s)}\right)^j \d{s} \\
&=:& I_{31}+I_{32}.
\end{eqnarray*}
Integrating by parts, we obtain in view of Lemma \ref{StirlingsDelta} that
\begin{eqnarray}\label{I_31f}
I_{31} 
&=& \dfrac{x^{1-\alpha}}{2\pi} \int_{T_2}^{T_1} \left(-\log\left(\lambda Q^2t^{d_\mathcal{L}}\right)+ O\left(\dfrac{1}{t}\right)\right)x^{it}\d{t} \nonumber\\
&=& \left.\dfrac{x^{1-\alpha+it}}{2\pi i\log x} \log\left(\lambda Q^2t^{d_\mathcal{L}} \right) \right|_{T_1}^{T_2}
- \dfrac{x^{1-\alpha}}{2\pi i\log x}\int_{T_1}^{T_2}O\left(\dfrac{1}{t}\right)\d{t} +O\left(x^{1-\alpha}\right)\nonumber\\
&=&\dfrac{x^{{1-}\alpha}}{2\pi i\log x}M(x,T_1,T_2)+O\left(x^{1-\alpha}\left(1+\dfrac{1}{\log x}\right)\right),
\end{eqnarray}
where $M(x,T_1,T_2)$ is as in \eqref{MainTerm}.
In the case of $I_{32}$, it follows from \eqref{SymmetryIdentity1} that
\begin{eqnarray*}
I_{32}
&=&-\dfrac{1}{2\pi i}\int_{\alpha-iT_2}^{\alpha-iT_1}\dfrac{\Delta_\mathcal{L}'}{\Delta_\mathcal{L}}(1-s)x^{1-s}\sum_{j\ge1} \left(\frac{a}{\Delta(1-s)}\right)^j \d{s}\nonumber\\
&=&\dfrac{1}{2\pi i}\int_{\alpha-iT_1}^{\alpha-iT_2}\overline{\dfrac{\Delta_\mathcal{L}'}{\Delta_\mathcal{L}}}(\overline{s})x^{1-s}\sum_{j\ge1} \left(a\overline{\Delta(\overline{s})}\right)^j \d{s}\nonumber\\
&=&-\dfrac{x^{1-\alpha}}{2\pi}\int_{T_1}^{T_2}\overline{\dfrac{\Delta_\mathcal{L}'}{\Delta_\mathcal{L}}}(\alpha+it)\sum_{j\ge1} \left(a\overline{\Delta_\mathcal{L}(\alpha+it)}\right)^j x^{it}\d{t}\nonumber
\end{eqnarray*}
or 
\begin{eqnarray}
\overline{I_{32}}&=&-\dfrac{x^{1-\alpha}}{2\pi}\int_{T_1}^{T_2}\dfrac{{\Delta_\mathcal{L}'(\alpha+it)}}{{\Delta_\mathcal{L}(\alpha+it)}}\sum_{j\ge1} \left({\overline{a}{\Delta_\mathcal{L}(\alpha+it)}}\right)^jx^{-it} \d{t}\nonumber\\
&=&-\dfrac{\overline{a}x^{1-\alpha}}{2\pi}\int_{T_1}^{T_2}\Delta_\mathcal{L}'(\alpha+it)\sum\limits_{j\geq0}\left(\overline{a}\Delta_\mathcal{L}(\alpha+it)\right)^j\left(\dfrac{1}{x}\right)^{it}\d{t}.
\end{eqnarray}
Now to estimate $\overline{I_{32}}$ we proceed exactly as \eqref{I_1ref} in the estimation of $I_1$, where we have $1/\overline{a}$ instead of $a$, $x^{1-\alpha}$ instead of $x^{\alpha}$ and $1/x$ instead of $x$ (as follows from (\ref{SymmetryIdentity1})). Then by the same reasoning as \eqref{I_1ref}--\eqref{sum1}, we finally obtain
\begin{equation}\label{I_3f}
I_{32}\ll x^{1-\alpha}(1+y(T)\log x),
\end{equation}
where we employed inequalities \eqref{inequality2} and our assumption that
$$1<\dfrac{1}{x}\leq\left(1-\dfrac{1}{y(T)}\right){\lambda Q^2T^{d_\mathcal{L}}}.$$

Collecting the estimates from \eqref{anf}, \eqref{hor}, \eqref{sum1}, \eqref{I_31f} and \eqref{I_3f} we conclude that
\begin{eqnarray}\label{firstvers}
\sum_{\substack{T<\gamma_a\leq T'}} x^{\delta_a}
&=&\dfrac{x^{{1-}\alpha}}{2\pi i\log x}M(x,T_1,T_2)+O\left(x^{\alpha}(\psi(T)+y(T)\log x)\right)+\nonumber\\
&&+~ O\left(x^{1-\alpha}\left(\dfrac{1}{\log x}+y(T)\log x\right)\right)
\end{eqnarray}
for any $T_0\leq T<T+1/\log T\leq T'\leq2T$ and $1<x\leq(1-1/y(T))\lambda Q^2T^{d_\mathcal{L}}$. 
Now observe that by our assumption on $x$ and $T_1,T_2$ we can deduce that
\[
x^{1-\alpha}=x^{1/2}\exp\left(\dfrac{-c\log x}{\log\left(\lambda Q^2T^{d_\mathcal{L}}\right)}\right)=x^{1/2}\left(1+O\left(\dfrac{\log x}{\log T}\right)\right)=x^\alpha
\]
and
\[
M(x,T_1,T_2)
=M(x,T,T')+O\left(\log x+\frac{1}{T\log T}\right)
\]
Therefore \eqref{firstvers} becomes
\begin{align*}
\sum_{\substack{T<\gamma_a\leq T'}} x^{\delta_a}=\dfrac{x^{1/2}M(x,T,T')}{2\pi i\log x}+O\left(x^{1/2}\left(\psi(T)+y(T)\log x+\dfrac{1}{\log x}\right)\right)
\end{align*}
for any $T_0\leq T<T+1/\log T\leq T'\leq2T$ and $1<x\leq(1-1/y(T))\lambda Q^2T^{d_\mathcal{L}}.$

The case of $1<1/x\leq(1-1/y(T))\lambda Q^2T^{d_\mathcal{L}}$ follows from the above relation and our observation in the end of the proof of Theorem \ref{RvMtheorem} that a complex number $z$ is an $a$-point of $\Delta_\mathcal{L}(s)$ (where $a\neq0$) if, and only if, the complex number $1-\overline{z}$ is a $b$-point of $\Delta_\mathcal{L}(s)$, where $b :=1/\overline{a}$ (see also \eqref{SymmetryIdentity1}). Hence for $0<x<1$, 
$$\dfrac{1}{x}\sum_{\substack{T<\gamma_a\leq T'}} x^{\delta_a}
= \sum_{\substack{T<\gamma_a\leq T'}}\left(\dfrac{1}{ x}\right)^{1-\delta_a}
= \overline{\sum_{\substack{T<\gamma_a\leq T'}}\left(\dfrac{1}{ x}\right)^{1-\overline{\delta_a}}}
= \overline{\sum_{\substack{T<\gamma_b\leq T'}}\left(\dfrac{1}{ x}\right)^{\delta_b}},$$
and thus
\[
\sum_{\substack{T<\gamma_a\leq T'}} x^{\delta_a}=\dfrac{x^{1/2}M(x,T,T')}{2\pi i(-\log x)}+O\left(x^{1/2}\left(\psi(T)-y(T)\log x-\dfrac{1}{\log x}\right)\right).
\]
\section{The Vertical Distribution of \texorpdfstring{$a$}{a}-points --- Proof of Corollary \ref{UDmod}}
For the proof of the corollary we need to introduce some notations.
Firstly, we will describe the $a$-points using two notations
$
\tilde{\beta}_a+i\tilde{\gamma}_a=\tilde{\delta}_a=\delta_a=\beta_a+i\gamma_a.
$
Moreover, $\delta_a'=\beta_a'+i\gamma_a'$ will denote the $a$-point whose imaginary part $\gamma_a'$ is the term succeeding $\gamma_a$ in the increasing sequence $(\gamma_a)_{\gamma_a>0}$.
Lastly, $\delta_a^T=\beta_a^T+i\gamma_a^T$ will denote the $a$-point whose imaginary part $\gamma_a^T$ is the last term of the sequence $(\gamma_a)_{\gamma_a>0}$ that is less than or equal to $2T$.

With the above notations in mind, we apply Abel's summation formula (summation by parts) to the sequence
\begin{eqnarray}\label{Abel}
\sum_{T<\gamma_a\leq 2T}x^{i\gamma_a}
&=&\sum_{T<\gamma_a\leq 2T}x^{-\beta_a}x^{\delta_a}\nonumber\\
&\ll&x^{-\beta_a^T}\left|\sum_{T<\gamma_a\leq2T}x^{\delta_a}\right|+\nonumber\\
&&+\mathop{\max_{T<\gamma_a\leq2T}}_{\gamma_a\neq\gamma_a^T}\left|\sum_{T<\tilde{\gamma}_a\leq\gamma_a}x^{\tilde{\delta}_a}\right|\mathop{\sum_{T<\gamma_a\leq2T}}_{\gamma_a\neq\gamma_a^T}\left|x^{-\beta_a'}-x^{-\beta_a}\right|.
\end{eqnarray}
Recall that 
\[
\beta_a=\dfrac{1}{2}-\dfrac{\log|a|}{\ell(\gamma_a)}+O\left(\dfrac{1}{\gamma_a\ell(\gamma_a)}\right),\quad\gamma_a>0,
\]
where $\ell(t)$ is defined in \eqref{abbreviation}, and 
$
|\log x|\leq\ell(T)
$
by our assumption \eqref{length}.
Therefore, we obtain for $\gamma_a>T$ that
\begin{equation}
x^{-\beta_a}=x^{-1/2}\exp\hspace{-1pt}\left(\dfrac{\log|a|\log x}{\ell(\gamma_a)}+O\left(\dfrac{|\log x|}{\gamma_a\ell(\gamma_a)}\right)\right)\ll x^{-1/2}\exp\hspace{-1pt}\left(\dfrac{|\log x|}{\ell(T)}\right)\ll x^{-1/2}
\end{equation}
and
\[
\beta_a'-\beta_a=\dfrac{\log|a|}{\ell(\gamma_a)}-\dfrac{\log|a|}{\ell(\gamma_a')}+O\left(\dfrac{1}{\gamma_a\ell(\gamma_a)}\right)\ll1.
\]
The above relations and the Riemann-von Mangoldt formula \eqref{RvM} imply that
\begin{eqnarray}\label{bound1}
\mathop{\sum_{T<\gamma_a\leq2T}}_{\gamma_a\neq\gamma_a^T}\left|x^{-\beta_a'}-x^{-\beta_a}\right|
&=&\mathop{\sum_{T<\gamma_a\leq2T}}_{\gamma_a\neq\gamma_a^T}x^{-\beta_a'}\left|x^{\beta_a'-\beta_a}-1\right|\nonumber\\
&\ll&x^{-1/2}\mathop{\sum_{T<\gamma_a\leq2T}}_{\gamma_a\neq\gamma_a^T}|\log x(\beta_a'-\beta_a)|\nonumber\\
&\ll&\dfrac{|\log x|}{x^{1/2}}\mathop{\sum_{T<\gamma_a\leq2T}}_{\gamma_a\neq\gamma_a^T}\left[\dfrac{1}{\ell(\gamma_a)}-\dfrac{1}{\ell(\gamma_a')}+O\left(\dfrac{1}{\gamma_a\ell(\gamma_a)}\right)\right]\nonumber\\
&\ll&\dfrac{|\log x|}{x^{1/2}}\left(\dfrac{1}{\ell(T)}-\dfrac{1}{\ell(2T)}+O(1)\right)\nonumber\\
&\ll&\dfrac{|\log x|}{x^{1/2}}.
\end{eqnarray}
The first statement \eqref{imagin} of the corollary  follows from the combination of Theorem \ref{Landau} and relations \eqref{Abel}-\eqref{bound1}.

Let now $\alpha\neq0$ be a real number and $k\neq0$ be an integer. 
If we set $x=\exp(2\pi\alpha k)$ and $y(t)=\psi(t)$, then for any sufficiently large $T>0$ we have that
\begin{eqnarray*}
\sum\limits_{T<\gamma_a\leq 2T}\exp\left(2\pi ik\alpha\gamma_a\right)
&=&\sum\limits_{T<\gamma_a\leq 2T}x^{i\gamma_{a}}\ll_x\log T
\end{eqnarray*}
as follows from the first statement  \eqref{imagin} of the corollary.
Thus, by splitting the interval $(0,T]$ in a dyadic manner $(T/2,T]$, $(T/4,T/2]$,..., we obtain for any sufficiently large $T>0$ that
\[
\sum\limits_{0<\gamma_a\leq T}\exp\left(2\pi ik\alpha\gamma_a\right)=O_x\left(\log T)^2\right)=o_x(T\log T).
\]
Since $k\in\mathbb{Z}\setminus\lbrace0\rbrace$ can be chosen arbitrarily, the sequence $\alpha\gamma_a$, $n\in\mathbb{N}$, satisfies Weyl's Criterion (see \cite{weyl}):
{\it a sequence $(x_n)_{n\in\mathbb{N}}$ is uniformly distributed modulo one if
	\begin{equation*}\sum\limits_{n\leq N}\exp(2\pi ikx_n)=o(N)
	\end{equation*}
	for any integer $k\neq0$}. This concludes the proof of the corollary.
\section{A Mean-Value Theorem --- Proof of Theorem \ref{moment}}


In this section, we estimate
$$ \sum_{0<\gamma_a\le T}{\mathcal L}(\delta_a) $$
for $\mathcal{L}\in\mathcal{S}^\sharp$ with $d_\mathcal{L}\geq1$.
It suffices to treat only the case $0<|a|\le1$ and deduce the result for $|a|>1$ from the functional equation of $\mathcal{L}(s)$ and our observation at the end of Section \ref{SectRvMtheorem}.

Let $a\neq0$ be such that $|a|\leq1$ and $T>0$.
In view of Theorem \ref{RvMtheorem}, we may assume that $T$ is in some distance to the imaginary parts $\gamma_a$ of the $a$-points, $\vert T-\gamma_a\vert\geq C/\log T$ for some positive constant $C$, say.
As in the previous sections, or more straightforwardly by Lemma \ref{Value-DistributionDelta} and \eqref{SymmetryIdentity2}, for some $t_0>0$ depending only on ${\mathcal L}$, all $a$-points $\delta_a=\beta_a+i\gamma_a$ of $\Delta_{\mathcal L}$ with $|\gamma_a|\ge t_0$, satisfying $0<|a|\le1$ lie in $\sigma\ge1/2$.
Thus, for any fixed $\varepsilon=\varepsilon(\mathcal{L})>0$, let ${\mathcal C}$ be the counter-clockwise oriented contour with vertices $2+iT_0,2+iT,1/2-\varepsilon+iT,1/2-\varepsilon+iT_0$ with $T_0=T_0(\varepsilon,a,\mathcal{L})$ such that
\begin{equation} \label{t_a}
\left(\lambda Q^2|t|^{d_{\mathcal L}}\right)^{-\varepsilon} \le \frac1{2|a|}
\end{equation}
holds for all $|t|\ge T_0\ge t_0$, by Cauchy's theorem, we have
\begin{eqnarray} \label{moment_cauchy}
\sum_{0<\gamma_a\le T} {\mathcal L}(\delta_a)
&=& {1\over 2\pi i}\int_{\mathcal C}{\Delta_{\mathcal L}'(s)\over \Delta_{\mathcal L}(s)-a}{\mathcal L}(s) \d s + O_\varepsilon(1).
\end{eqnarray}
%
We remark that the left side of the contour ${\mathcal C}$ is taken to be as close as possible to $\sigma=1/2$, instead of any arbitrary line $\sigma=\sigma_0<1/2$ to overcome the growth of ${\mathcal L}(\sigma+it)$ as $\sigma$ decreases.
We further remark that the error term in \eqref{moment_cauchy} relies on the $a$-points $\delta_a=\beta_a+i\gamma_a$ to be considered in the sum on the left but excluded by the contour of integration on the right.

Before we begin, if we let $A\ge0$ be such that
$$ 
{\mathcal L}(1/2+it) \ll_\epsilon |t|^{A+\epsilon} \qquad \text{as}~ t\to\pm\infty, 
$$
then it follows from the Phragm\'en-Lindel\"of principle (see in \cite[\S 9.41]{titch}) that
\begin{equation} \label{esti}
{\mathcal L}(\sigma+it) \ll_\epsilon \left\{\begin{array}{ll}
t^{{d_\mathcal{L}}/2+(2A-d_\mathcal{L})\sigma+\epsilon} & \mbox{for}\quad 0\leq \sigma\leq 1/2,\\
t^{2A(1-\sigma)+\epsilon} & \mbox{for}\quad 1/2\le\sigma\leq 1,\\
t^\epsilon & \mbox{for}\quad \sigma\geq 1,\\
\end{array}\right.
\end{equation}
as $t\to\infty$.
For several $L$-functions, we have subconvexity bounds which may be used in place of $A$ to deduce a sharper bound.
In this paper, however, we leave it as a constant $A$ to include all general cases.

We begin our evaluations with the integral over the line segment within the half-plane of absolute convergence of ${\mathcal L}(s)$. Since
$$
{\Delta_{\mathcal L}'(s)\over \Delta_{\mathcal L}(s)-a}=\dfrac{\Delta_{\mathcal L}'}{\Delta_{\mathcal L}}(s)\cdot {1\over 1-a/\Delta_{\mathcal L}(s)}\ll t^{d_{\mathcal L}(1/2-\sigma)}\log t,
$$
applying Lemma \ref{StirlingsDelta} and \eqref{esti} gives
\begin{equation} \label{right_int}
\int_{2+iT_0}^{2+iT}{\Delta_{\mathcal L}'(s)\over \Delta_{\mathcal L}(s)-a}{\mathcal L}(s) \d s\ll _\epsilon\int_{T_0}^T t^{-3d_{\mathcal L}/2+\epsilon}\d t \ll_\epsilon 1.
\end{equation}

Next we consider the horizontal integrals. Taking into account Lemma \ref{fractionaldecomposition}, bounds \eqref{esti} and that $T$ is such that $\vert T-\gamma_a\vert\gg 1/\log T$, the integral over the upper horizontal line segment can be estimated as
\begin{eqnarray} \label{hori}
\lefteqn{\int_{1/2-\varepsilon+iT}^{2+iT} {\Delta_{\mathcal L}'(s)\over \Delta_{\mathcal L}(s)-a}{\mathcal L}(s) \d s \nonumber}\\
&= &\int_{1/2-\varepsilon+iT}^{2+iT} \left(\sum\limits_{|T-\gamma_a|\leq \frac1{\log\log{T}}}\dfrac{1}{s-\delta_a} + O\left(\log{T}\right) \right) {\mathcal L}(s) \d s \nonumber\\
&\ll_\epsilon& T^{A+(d_{\mathcal L}-2A)\varepsilon+\epsilon}\, (\log{T})^2 \int_{1/2-\varepsilon}^{2} 1 \d\sigma \nonumber\\
&\ll_\epsilon& T^{A+(d_{\mathcal L}-2A)\varepsilon+\epsilon}.
\end{eqnarray}
Meanwhile, the lower horizontal line segment is independent of $T$ and thus bounded.

It remains to evaluate the integral over the left vertical line segments
\begin{equation}\label{J}
J:=-{1\over 2\pi i} \int_{1/2-\varepsilon+iT_0}^{1/2-\varepsilon+iT} {\Delta_{\mathcal L}'(s)\over \Delta_{\mathcal L}(s)-a}{\mathcal L}(s) \d s.
\end{equation}
Applying (\ref{delta_i3_rewrite}) we have
\begin{eqnarray*}
J &=& -{1\over 2\pi i} \int_{1/2-\varepsilon+iT_0}^{1/2-\varepsilon+iT} \dfrac{\Delta_{\mathcal L}'}{\Delta_{\mathcal L}}(s) \sum_{k\geq 0} \left(\dfrac{a}{\Delta_{\mathcal L}(s)}\right)^k {\mathcal L}(s) \d s \\
&=& -{1\over 2\pi i} \sum_{k\geq 0} a^k \int_{1/2-\varepsilon+iT_0}^{1/2-\varepsilon+iT} \dfrac{\Delta_{\mathcal L}'}{\Delta_{\mathcal L}}(s) \left(\dfrac1{\Delta_{\mathcal L}(s)}\right)^k {\mathcal L}(s) \d s.
\end{eqnarray*}
Here we may interchange the order of integration and summation, since the sum is absolutely convergent.
Next we shift the contour to the right, but we will handle it differently for $k=0$ and $k\ge1$.

We first consider the case $k=0$,
$$
J_1 := -{1\over 2\pi i} \int_{1/2-\varepsilon+iT_0}^{1/2-\varepsilon+iT} \dfrac{\Delta_{\mathcal L}'}{\Delta_{\mathcal L}}(s) {\mathcal L}(s) \d s.
$$
Since the integrand is analytic on the upper half-plane, we may replace the integral above with
$$
\left\{\int_{1/2-\varepsilon+iT_0}^{2+iT_0}+\int_{2+iT_0}^{2+iT}+\int_{2+iT}^{1/2-\varepsilon+iT}\right\}\dfrac{\Delta_{\mathcal L}'}{\Delta_{\mathcal L}}(s) {\mathcal L}(s) \d s.
$$
As in \eqref{hori}, the integrals over the horizontal line segments are bounded by $T^{A+(d_{\mathcal L}-2A)\varepsilon+\epsilon}$ since Lemma \ref{StirlingsDelta} implies that
\begin{equation} \label{shift_right_hori1}
\dfrac{\Delta_{\mathcal L}'}{\Delta_{\mathcal L}}(\sigma+iT) \asymp \log{T}, \qquad \text{for } \sigma\in[0,2]
\end{equation}
while the bound \eqref{esti} yields
\begin{equation} \label{shift_right_hori2}
\mathcal{L}(\sigma+iT) \ll_\epsilon T^{A+(d_{\mathcal L}-2A)\varepsilon+\epsilon}, \qquad \text{for } \sigma\geq 1/2-\varepsilon.
\end{equation}
To estimate the integral on the line $\sigma=2$, we use the Dirichlet series representation for ${\mathcal L}(s)$.
Taking into account Lemma \ref{StirlingsDelta}, and interchanging integration and summation (which is allowed by the absolute convergence of the series), we obtain
\begin{eqnarray*}
J_1 &=& -{1\over 2\pi} \int_{T_0}^{T} \dfrac{\Delta_{\mathcal L}'}{\Delta_{\mathcal L}}(2+it){\mathcal L}(2+it)\d t + O_\epsilon\left(T^{A+(d_{\mathcal L}-2A)\varepsilon+\epsilon}\right) \\
&=& {1\over 2\pi}\sum_{n\geq 1} \frac{f(n)}{n^2} \int_{T_0}^T \left(\log\Big(\lambda Q^2t^{d_{\mathcal L}}\Big) + O\left(\frac1t\right)\right) \exp(-it\log n)\d t \\
&&+~ O_\epsilon\left(T^{A+(d_{\mathcal L}-2A)\varepsilon+\epsilon}\right).
\end{eqnarray*}
The constant term $f(1)$ of the Dirichlet series yields a contribution to the main term while for $n\geq 2$, a straightforward application of the first derivative test \eqref{firstderivativetest} gives 
$$
\int_{T_0}^T \left(\log\Big(\lambda Q^2t^{d_{\mathcal L}}\Big) + O\left(\frac1t\right)\right) \exp(-it\log n)\d t \ll {\log T\over \log n}+\log T,\quad n\geq2.
$$
Hence,
\begin{equation}\label{j1}
J_1 = f(1){T\over 2\pi}\log\left(\lambda Q^2\Big({T\over e}\Big)^{d_{\mathcal L}}\right) + O_\epsilon\left(T^{A+(d_{\mathcal L}-2A)\varepsilon+\epsilon}\right).
\end{equation}

Now when $k\ge1$, that is for
$$
J_2 := -{1\over 2\pi i} \sum_{k\geq 1} a^k \int_{1/2-\varepsilon+iT_0}^{1/2-\varepsilon+iT} \dfrac{\Delta_{\mathcal L}'}{\Delta_{\mathcal L}}(s) \left(\dfrac1{\Delta_{\mathcal L}(s)}\right)^k {\mathcal L}(s) \d s,
$$
the situation gets slightly more complicated.
We first rewrite the integrand as
\begin{eqnarray*}
	\dfrac{\Delta_{\mathcal L}'}{\Delta_{\mathcal L}}(s) \left(\dfrac1{\Delta_{\mathcal L}(s)}\right)^k {\mathcal L}(s)
	&=&	\dfrac{\Delta_{\mathcal L}'}{\Delta_{\mathcal L}}(s) \left(\dfrac1{\Delta_{\mathcal L}(s)}\right)^{k-1} \overline{{\mathcal L}(1-\overline{s})}\\
	&=& \overline{\dfrac{\Delta_{\mathcal L}'}{\Delta_{\mathcal L}}(1-\overline{s}) \left({\Delta_{\mathcal L}(1-\overline{s})}\right)^{k-1} {\mathcal{L}(1-\overline{s})}}
\end{eqnarray*}
using the functional equations (for ${\mathcal L}(s)$ as well as the one for the logarithmic derivative of $\Delta_{\mathcal L}(s)$ which follows directly from \eqref{SymmetryIdentity1}).
Therefore,
$$
\overline{J_2} =- {1\over 2\pi i} \sum_{k\geq 1} \overline{a}^k \int_{1/2+\varepsilon+iT_0}^{1/2+\varepsilon+iT}
{\dfrac{\Delta_{\mathcal L}'}{\Delta_{\mathcal L}}(s)} \Delta_{\mathcal L}(s)^{k-1} \mathcal{L}(s) \d s.
$$
We now shift the line of integration to the right and get
\begin{equation} \label{shift_right}
\overline{J_2} = -{1\over 2\pi i} \sum_{k\geq 1} \overline{a}^k \left\{\int_{1/2+\varepsilon+iT_0}^{2+iT_0}+\int_{2+iT_0}^{2+iT}+\int_{2+iT}^{1/2+\varepsilon+iT}\right\} \dfrac{\Delta_{\mathcal L}'}{\Delta_{\mathcal L}}(s) \Delta_{\mathcal L}(s)^{k-1} {\mathcal L}(s) \d s.
\end{equation}
We again recall that Lemma \ref{StirlingsDelta} implies that for $\sigma\in[1/2+\varepsilon,2]$,
$$ \Delta_{\mathcal L}(s) \ll \left(\lambda Q^2|t|^{d_\mathcal{L}}\right)^{-\varepsilon} \le \dfrac1{2|a|} $$
by \eqref{t_a}.
Thus, the infinite sum of the horizontal integrals can again be bounded using \eqref{shift_right_hori1} and \eqref{shift_right_hori2} as
\begin{eqnarray*}
\sum_{k\geq 1} \overline{a}^k \int_{1/2+\varepsilon+iT}^{2+iT}
\dfrac{\Delta_{\mathcal L}'}{\Delta_{\mathcal L}}(s) \Delta_{\mathcal L}(s)^{k-1} {\mathcal L}(s)
&\ll_\epsilon& T^{A+(d_{\mathcal L}-2A)\varepsilon+\epsilon} \sum_{k\geq 1} |a|^k \left(\dfrac1{2|a|}\right)^{k-1} \\
&\ll_\epsilon& T^{A+(d_{\mathcal L}-2A)\varepsilon+\epsilon} \sum_{k\geq 1} \left(\dfrac1{2}\right)^{k-1} \\
&\ll_\epsilon& T^{A+(d_{\mathcal L}-2A)\varepsilon+\epsilon},
\end{eqnarray*}
for any $T\geq T_0$. Thus we arrive at
\begin{eqnarray*}
\overline{J_2} &=& -{1\over 2\pi i} \sum_{k\geq 1} \overline{a}^k \int_{2+iT_0}^{2+iT} \dfrac{\Delta_{\mathcal L}'}{\Delta_{\mathcal L}}(s) \Delta_{\mathcal L}(s)^{k-1} {\mathcal L}(s) \d s
+ O_\epsilon\left(T^{A+(d_{\mathcal L}-2A)\varepsilon+\epsilon}\right) \\
&=& {1\over 2\pi} \sum_{k\geq 1} \overline{a}^k \int_{T_0}^{T}
\left( \log\left(\lambda Q^2t^{d_\mathcal{L}}\right) + O\left( \dfrac{\log{t}}{t} \right) \right)
\Delta_{\mathcal L}(2+it)^{k-1} {\mathcal L}(2+it) \d t \\
&&+~ O_\epsilon\left(T^{A+(d_{\mathcal L}-2A)\varepsilon+\epsilon}\right) \\
&=& {1\over 2\pi} \sum_{k\geq 1} \overline{a}^k \sum_{n=1}^\infty \dfrac{f(n)}{n^2}
\int_{T_0}^{T} \log\left(\lambda Q^2t^{d_\mathcal{L}}\right) \exp(-it\log{n}) \Delta_{\mathcal L}(2+it)^{k-1} \d t \\
&&+~ O_\epsilon\left(T^{A+(d_{\mathcal L}-2A)\varepsilon+\epsilon}\right),
\end{eqnarray*}
where we again used the Dirichlet series representation for ${\mathcal L}(s)$ and interchanged integration and summation.
As in the estimation of $J_1$, only the constant term $f(1)$ of the Dirichlet series yields a contribution to the main term when $k=1$, while for $k\ge2$, $\Delta_{\mathcal L}(2+it)$ is very small (recall Lemma \ref{StirlingsDelta}) so that
\begin{eqnarray*}
\lefteqn{\dfrac{1}{2\pi}\sum_{k\geq 2} \overline{a}^k \sum_{n=1}^\infty \dfrac{f(n)}{n^2} \int_{T_0}^{T} \log\left(\lambda Q^2t^{d_\mathcal{L}}\right) \exp(-it\log{n}) \Delta_{\mathcal L}(2+it)^{k-1} \d t} \\
	&\ll& \sum_{k\geq 1} |{a}|^k \sum_{n=1}^\infty \dfrac{|f(n)|}{n^2}
	\int_{T_0}^{T} \left(\lambda Q^2t^{d_\mathcal{L}}\right)^{-3k/2} \log{t} \d t \hspace{33mm}\\
	&\ll& \sum_{k\geq 1} |{a}|^k \left(\lambda Q^2T_0^{d_\mathcal{L}}\right)^{-3k/2} T_0\log{T_0} \\
	&\ll& 1
\end{eqnarray*}
for $T_0>0$ sufficiently large (recall \eqref{t_a}).
Hence we obtain
\begin{eqnarray} \label{j2}
J_2 = a\overline{f(1)}{T\over 2\pi}\log\left(\lambda Q^2\Big({T\over e}\Big)^{d_{\mathcal L}}\right)
+ O_\epsilon\left(T^{A+(d_{\mathcal L}-2A)\varepsilon+\epsilon}\right).
\end{eqnarray}
Combining \eqref{j1} and \eqref{j2} we deduce
\begin{eqnarray}\label{Jfinal}
J &=& J_1+J_2 \nonumber\\
&=& (f(1)+a\overline{f(1)}){T\over 2\pi}\log\left(\lambda Q^2\left({T\over e}\right)^{d_{\mathcal L}}\right) 
+ O_\epsilon\left(T^{A+(d_{\mathcal L}-2A)\varepsilon+\epsilon}\right).
\end{eqnarray}

Substituting \eqref{right_int}, \eqref{hori}, \eqref{J} and \eqref{Jfinal} into \eqref{moment_cauchy}, 
and rewriting the main term with that of Theorem \ref{RvMtheorem}, we arrive at
\begin{eqnarray}\label{final1}
\sum_{0<\gamma_a\le T} {\mathcal L}(\delta_a)
= (f(1)+a\overline{f(1)})N_+(T;a,\Delta_\mathcal{L})
+ O_\epsilon\left(T^{A+(d_{\mathcal L}-2A)\varepsilon+\epsilon}\right).
\end{eqnarray}
Since $\varepsilon>0$ can be taken arbitrarily small depending only on $\mathcal{L}$, the error term is bounded from above by
$O_\epsilon (T^{A+\epsilon}). $

If $T$ is not in some distance to all $a$-points, then we choose $T'\in[T-1,T]$ to be such a number.
In view of Theorem \ref{RvMtheorem}, our assumption on the order of $\mathcal{L}(1/2+it)$ and \eqref{final1} yields
\begin{eqnarray*}
\sum_{0<\gamma_a\le T} {\mathcal L}(\delta_a)
&=& \sum_{0<\gamma_a\le T'} {\mathcal L}(\delta_a) + \sum_{T'<\gamma_a\le T} {\mathcal L}(\delta_a) \\
&=& (f(1)+a\overline{f(1)})N_+(T';a,\Delta_\mathcal{L})
+ O_\epsilon\left((T')^{A+\epsilon}\right)+O_\epsilon(T^{A+\epsilon})\\
&=& (f(1)+a\overline{f(1)})N_+(T;a,\Delta_\mathcal{L})
+ O_\epsilon\left(T^{A+\epsilon}\right)
\end{eqnarray*}

Lastly if $|a|>1$, we set $b=1/\overline{a}$ and we recall that a complex number $z$ is an $a$-point of $\Delta_\mathcal{L}(s)$ if, and only if, the complex number $1-\overline{z}$ is a $b$-point of $\Delta_\mathcal{L}(s)$. Thus using the functional equation for $\mathcal{L}(s)$ and employing \eqref{final1} for $0<|b|<1$, we have
\begin{eqnarray*}
	\sum_{\substack{0<\gamma_a\le T}} {\mathcal L}(\delta_a)
	&=& \sum_{\substack{0<\gamma_a\le T}} \Delta_{\mathcal L}(\delta_a) \overline{{\mathcal L}(1-\overline{\delta_a})} \\
	&=& a \sum_{\substack{0<\gamma_b\le T}} \overline{{\mathcal L}(\delta_b)}\\
	&=& a\overline{(f(1)+b\overline{f(1)})N_+(T;b,\mathcal{L})} + O_\epsilon\left(T^{A+\epsilon}\right) \\
	&=& (a\overline{f(1)}+f(1)) N_+(T;a,\mathcal{L}) + O_\epsilon\left(T^{A+\epsilon}\right).
\end{eqnarray*}
This concludes the proof of the theorem.

\section{Concluding Remarks}
As we already mentioned in the introduction, one could undergo the same investigations besides the extended Selberg class for other classes of Dirichlet series. In view of Hecke's theory \cite{hecke} one could as well consider a Riemann-type functional equation where on the right hand side conjugation is dropped (as they appear in context of the Hecke groups generated by $\tau\mapsto \tau+\lambda$ and $\tau\mapsto -1/\tau$ for values of positive $\lambda\neq 1$); for these modified Selberg classes we refer to Kaczorowski et al. \cite{kac}. Further, there are Dirichlet series satisfying other variations of the Riemann functional equation, e.g., Dirichlet series associated with even or odd periodic arithmetical functions, as investigated in \cite{soursteusur}. 
The same methods, however, may be applied to obtain analogous results. 

We conclude with a historical remark. The very origin of Riemann-type functional equations lies already in the work of Leonhard Euler (whose investigations were limited to the real ine). For a historical account and interesting details we refer to Andre Weil \cite{weil} and Nicola Oswald \cite{nico}.

\appendix

\section{Quotients of Gamma Factors}
\renewcommand{\theequation}{A.\arabic{equation}}
\setcounter{equation}{0}

Since $\Delta_\mathcal{L}(s)$ is mainly a product of quotients of Gamma factors, Stirling's formula plays a prominent role in the study of its value-distribution:
\begin{equation}\label{Stirling'sformula2}
\log\Gamma(s)=\left(s-\dfrac{1}{2}\right)\log s-s+\dfrac{1}{2}\log2\pi+\dfrac{1}{12s}-2\int_0^\infty\dfrac{P_3(x)}{(s+x)^3}\mathrm{d}x;
\end{equation}
here $\log s$ is the principal branch of the logarithm and $P_3(x)$ is a $1$-periodic function which is equal to $x\left(2x^2-3x+1\right)/12$ on $[0,1]$
(see for example \cite[Chapter XV, \S 2, pp. 422--430]{lang}).
Notice that, for any $0<\epsilon\leq\pi$, the above integral is $O_\epsilon\left(|s|^{-2}\right)$ in the angular space $W_\epsilon:=\left\{s\in\mathbb{C}:|\mathrm{arg} s|\leq\pi-\epsilon\right\}$.
Hence, the following formulas hold in $W_\epsilon$:
\begin{eqnarray}
\label{StirlingGamma1}\log\Gamma(s)&=&\left(s-\dfrac{1}{2}\right)\log s-s+\dfrac{1}{2}\log2\pi+O_\epsilon\left(\dfrac{1}{|s|}\right),\\
\label{StirlingGamma2}\dfrac{\Gamma'}{\Gamma}(s)&=&\log s-\dfrac{1}{2s}+O_\epsilon\left(\dfrac{1}{|s|^2}\right).
\end{eqnarray}
We employ these formulas to prove the following lemma.
A similar result was given in \cite[Lemma 6.7]{steudi}, but it is unfortunately incorrect.
Here we give a more rigorous proof and obtain the correct statement.

\begin{Lem}[Stirling's formula]\label{StirlingsDelta}
If $\sigma$ is from an interval of bounded width, then
\begin{equation*}
\Delta_\mathcal{L}(\sigma+it)=\left(\lambda Q^2t^{d_\mathcal{L}}\right)^{1/2-\sigma-it}\exp\left(i\left(d_\mathcal{L}t-\Theta\log t\right)\right)\omega^*\left(1+O\left(\dfrac{1}{t}\right)\right)
\end{equation*}
for any $t\geq t_0>0$,
where 
\[
\omega^*:=\omega\exp\left(-\frac{\pi i}{4}\left({d_\mathcal{L}}+2\eta\right)\right)\prod\limits_{j\leq r}\lambda_j^{-2i\Im\mu_j}
\]
is a complex number from the unit circle and $d_\mathcal{L},\lambda$ and $\eta+i\Theta$ are defined in \eqref{dlm}.
In addition,
\begin{equation*}
\dfrac{\Delta_\mathcal{L}'}{\Delta_\mathcal{L}}(\sigma+it)=-\log\left(\lambda Q^2|t|^{d_\mathcal{L}}\right)-\dfrac{\Theta}{t}+\dfrac{d_\mathcal{L}\left(\frac{1}{2}-\sigma\right)i}{t}+O\left(\dfrac{1}{t^2}\right)
\end{equation*}
for any $|t|\geq t_0>0$.
\end{Lem}

\noindent The asymptotics of this lemma imply that all $a$-points $\delta_a=\beta_a+i\gamma_a$ with sufficiently large $\vert \gamma_a\vert$ are simple (since $\Delta_{\mathcal L}'$ is vanishing only in neighbourhood of the real axis). Hence, the Riemann-von Mangoldt formula from Theorem \ref{RvMtheorem} counts simple $a$-points. 

\begin{proof}
Formula \eqref{StirlingGamma1} implies that
$$\Gamma(\sigma+it)=\sqrt{2\pi}t^{\sigma+it-1/2}\exp\left(-\dfrac{\pi t}{2}-it+\dfrac{\pi i}{2}\left(\sigma-\dfrac{1}{2}\right)\right)\left(1+O\left(\dfrac{1}{t}\right)\right)$$
 for any $\sigma$ from an interval of bounded width and any $t\geq t_0>0$. 
With the aid of this expression and the reflection property of the Gamma function $\Gamma(\sigma-it)=\overline{\Gamma(\sigma+it)}$,
we can now rewrite $\Delta_\mathcal{L}(s)$ as follows
\begin{eqnarray}
\Delta_\mathcal{L}(s)&=&\omega Q^{1-2s}\prod\limits_{j\leq r}\dfrac{\Gamma\left(\lambda_j(1-s)+\overline{\mu_j}\right)}{\Gamma\left(\lambda_js+\mu_j\right)}\nonumber\\
&=&\omega Q^{1-2s}\prod\limits_{j\leq r}\dfrac{\overline{\Gamma\left(\lambda_j(1-\sigma)+\Re{\mu_j}+i\left(\lambda_jt+\Im\mu_j\right)\right)}}{\Gamma\left(\lambda_j\sigma+\Re\mu_j+i\left(\lambda_jt+\Im\mu_j\right)\right)}\nonumber\\
&=&\omega Q^{1-2s}\prod\limits_{j\leq r}\left(\lambda_jt+\Im\mu_j\right)^{\lambda_j(1-2\sigma)-2i\left(\lambda_jt+\Im\mu_j\right)}\times\nonumber\\
&&\times\exp\left(2i\left(\lambda_jt+\Im\mu_j\right)-\dfrac{\pi i}{2}\left(\lambda_j+2\Re\mu_j-1\right)\right)\left(1+O\left(\dfrac{1}{t}\right)\right)\nonumber.
\end{eqnarray}
Here we note that if $a>0$, $b$ and $c$ are fixed real numbers, then
\begin{eqnarray*}
(at+b)^{c-2i(at+b)}&=&\exp\left(\left(c-2i(at+b)\right)\log(at+b)\right)\\
&=&\exp\left(\left(c-2i(at+b)\right)\left(\log(at)+\frac{b}{at}+O\left(\dfrac{1}{t^2}\right)\right)\right)\\
&=&\exp\left(\left(c-2i(at+b)\right)\log(at)-2ib+O\left(\dfrac{1}{t}\right)\right)\\
&=&(at)^{c-2i(at+b)}e^{-2ib}\left(1+O\left(\dfrac{1}{t}\right)\right)
\end{eqnarray*}
for any $t\geq t_0>0$, where the implicit constant in the $O$-term depends on $a,b$ and $c$.
Therefore,
\begin{eqnarray*}
\Delta_\mathcal{L}(s)&=&\omega\left(Q^2\right)^{1/2-\sigma-it}\prod\limits_{j\leq r}(\lambda_jt)^{2\lambda_j(1/2-\sigma-it)-2i\Im\mu_j}e^{-2i\Im\mu_j}\times\nonumber\\
&&\times\exp\left(2i\left(\lambda_jt+\Im\mu_j\right)-\dfrac{\pi i}{2}\left(\lambda_j+2\Re\mu_j-1\right)\right)\left(1+O\left(\dfrac{1}{t}\right)\right)\nonumber\\
&=&\left(\lambda Q^2t^{d_\mathcal{L}}\right)^{1/2-\sigma-it}\exp\left(i\left(d_\mathcal{L}t-\Theta\log t\right)\right)\omega^*\left(1+O\left(\dfrac{1}{t}\right)\right)
\end{eqnarray*}
for any $\sigma$ from an interval of bounded width and any $t\geq t_0>0$.

Formula \eqref{StirlingGamma2} implies for any $\sigma$ from an interval of bounded width and any $|t|\geq t_0>0$ that
\begin{eqnarray}\label{logarithmicderiv}
\dfrac{\Delta'_\mathcal{L}}{\Delta_\mathcal{L}}(s)&=&-2\log Q-\sum\limits_{j\leq r}\lambda_j\left(\dfrac{\Gamma'}{\Gamma}\left(\lambda_j(1-s)+\overline{\mu_j}\right)+\dfrac{\Gamma'}{\Gamma}\left(\lambda_js+{\mu_j}\right)\right)\nonumber\\
&=&-2\log Q-\sum\limits_{j\leq r}\lambda_j\left(\log\left(\lambda_j(1-s)+\overline{\mu_j}\right)+\log\left(\lambda_js+{\mu_j}\right)\right)+\nonumber\\
&&+~ \dfrac{1}{2}\sum\limits_{j\leq r}\lambda_j\left(\dfrac{1}{\lambda_j(1-s)+\overline{\mu_j}}+\dfrac{1}{\lambda_js+{\mu_j}}\right)+O_{}\left(\dfrac{1}{t^2}\right)\nonumber\\
&=:&-2\log Q-\sum\limits_{j\leq r}\lambda_jF_j(s)+\dfrac{1}{2}\sum\limits_{j\leq r}\lambda_jG_j(s)+O_{}\left(\dfrac{1}{t^2}\right).
\end{eqnarray}
Treating each summand seperately, we obtain for every $j=1,\dots,r$ and $|t|\geq t_0$ that
\begin{eqnarray}
F_j(s)
&=&\log\left(\lambda_j(1-s)\right)+\log\left(1+\dfrac{\overline{\mu_j}}{\lambda_j(1-s)}\right)+\log\left(\lambda_js\right)+\nonumber\\
&&+~ \log\left(1+\dfrac{\mu_j}{\lambda_js}\right)\nonumber\\
&=&\log\left(\lambda_j(1-s)\right)+\dfrac{\overline{\mu_j}}{\lambda_j(1-s)}+\log\left(\lambda_js\right)+\dfrac{\mu_j}{\lambda_js}+O_{}\left(\dfrac{1}{t^2}\right)\nonumber\\
&=&\log\left|\lambda_j(1-s)\right|+\log\left|\lambda_js\right|+i\arg\left(s\right)+i\arg\left(1-s\right)+\nonumber\\
&&+~ \dfrac{2i\mu_j}{\lambda_j(s-1)}+O\left(\dfrac{1}{t^2}\right)\nonumber\\
&=&2\log\lambda_j+\log|1-s|+\log|s|+\dfrac{2\mu_j}{\lambda_jt}+\nonumber\\
&&-~ i\left(\arctan\left(\dfrac{\sigma}{t}\right)+\arctan\left(\dfrac{1-\sigma}{-t}\right)\right)+O_{}\left(\dfrac{1}{t^2}\right)\nonumber\\
&=&2\log\lambda_j+2\log |t|+\dfrac{2\mu_j}{\lambda_jt}+\dfrac{(1-2\sigma)i}{t}+O_{}\left(\dfrac{1}{t^2}\right)
\end{eqnarray}
and
\begin{eqnarray}\label{Gj}
G_j(s)
=\dfrac{\mu_j+\overline{\mu_j}+\lambda_j}{\left(\lambda_j(1-s)+\overline{\mu_j}\right)\left(\lambda_js+{\mu_j}\right)}
\ll_{}\dfrac{1}{t^2}.
\end{eqnarray}
It now follows from \eqref{logarithmicderiv}--\eqref{Gj} that
\begin{eqnarray*}
\dfrac{\Delta'_\mathcal{L}}{\Delta_\mathcal{L}}(s)
=-\log\left(\lambda Q^2|t|^{d_\mathcal{L}}\right)-\dfrac{\Theta}{t}+\dfrac{d_\mathcal{L}\left(\frac{1}{2}-\sigma\right)i}{t}+O\left(\dfrac{1}{t^2}\right)
\end{eqnarray*}
for any $\sigma$ from an interval of bounded width and $|t|\geq t_0>0$.
\end{proof}

The next lemma originates from the work of Robert Spira \cite{spira}, Robert Dan Dixon \& Lowell Schoenfeld \cite{dixsch} and Berndt \cite{berndt}. The proof is very similar to the one provided from the last one \cite[Theorem 12]{berndt} and
we only give a brief sketch of it here by adding a few new details where needed. 

\begin{Lem}\label{Value-DistributionDelta}
There exists $t_0=t_0\left(\mathcal{L}\right)>0$ such that
\begin{eqnarray*}
\left|\Delta_\mathcal{L}(s)\right|<1,\qquad\left|\Delta_\mathcal{L}(1-s)\right|>1\qquad\text{and}\qquad\left|\Delta_\mathcal{L}\left(\dfrac{1}{2}+it\right)\right|=1
\end{eqnarray*}
 for any $|t|\geq t_0>0$ and $\sigma>1/2$.
\end{Lem}

\begin{proof}
If
$$\Delta(s):=\prod\limits_{j\leq r}\Gamma\left(\lambda_js+\mu_j\right),$$
then
$$\Delta_\mathcal{L}(s)=\omega Q^{1-2s}\dfrac{\overline{\Delta\left(1-\overline{s}\right)}}{\Delta(s)}.$$
For $t\neq0$, the function $\Delta_\mathcal{L}(s)$ is analytic and non-zero.
If we define in the upper half-plane the function
$$h(s)=-\log\left|\Delta_\mathcal{L}(s)\right|,$$
then it suffices to prove that $h(s)>0$ since 
\begin{equation}\label{SymmetryIdentity2}
\Delta_\mathcal{L}(s)\overline{\Delta_\mathcal{L}\left(1-\overline{s}\right)}=1\qquad\text{and}\qquad\Delta_\mathcal{L}(\overline{s})=\overline{\Delta_\mathcal{\overline{L}}\left(s\right)}.
\end{equation}
Observe that
\begin{equation}\label{MoreGeneralCase}
h(s)=(2\sigma-1)\log Q+\log|\Delta(\sigma+it)|-\log|\Delta(1-\sigma+it)|
\end{equation}
By the mean value theorem, there exists $\sigma_1$ with $1/2<1-\sigma<\sigma_1<\sigma$ such that 
\begin{equation}\label{Berndt}
{\log|\Delta(\sigma+it)|-\log|\Delta(1-\sigma+it)|}=(2\sigma-1)\left.\dfrac{\partial}{\partial x}\log|\Delta(x+it)|\right|_{x=\sigma_1}.
\end{equation}
For real numbers $\mu_j$, it is proved in \cite{berndt} that the right-hand side of the relation above is greater than $(2\sigma-1)\Sigma(s_1)$, 
where $s_1:=\sigma_1+it$ and $\Sigma(s_1)$ is defined by
\begin{equation*}
\sum\limits_{j\leq r}\lambda_j\left(\log|\lambda_j s_1+\mu_j|-\dfrac{1}{2|\lambda_j s_1+\mu_j|}-\dfrac{1}{12|\lambda_j s_1+\mu_j|^2}-\dfrac{\pi}{16\lambda_j^3|t|^3}\right).
\end{equation*}
One of the main arguments in the proof is to expand each factor $\log\Gamma\left(\lambda_js+\mu_j\right)$ as in \eqref{Stirling'sformula2}.
Observe that it does not matter whether $\mu_j$ are complex or real numbers, as long as $\lambda_j\neq0$, since for sufficiently large $t$, depending on fixed and finitely many $\lambda_j$ and $\mu_j$, the arguments in \cite[Theorem 12]{berndt} remain the same.
As a matter of fact, we prove a more general case. 
It is now obvious that
$$\Sigma(s_1)\gg\log t$$
for any sufficietly large $t>1$ and $\sigma_1\geq1/2$.
In combination with \eqref{MoreGeneralCase} and \eqref{Berndt}, we obtain that $h(s)>0$ for any sufficiently large $t>1$ and any $\sigma>1/2$.
\end{proof}
An easy consequence of Lemma \ref{StirlingsDelta} and identity \eqref{SymmetryIdentity2} is that
\begin{eqnarray}\label{Deltasymp}
\left|\Delta_\mathcal{L}(\sigma+it)\right|=\left(\lambda Q^2|t|^{d_\mathcal{L}}\right)^{1/2-\sigma}\left(1+O\left(\dfrac{1}{|t|}\right)\right)
\end{eqnarray}
for any $\sigma$ from an interval of bounded width and $|t|\geq t_0>0$.
We will use this relation to prove

\begin{Lem}\label{fractionaldecomposition}
Let $a\neq0$ be a complex number. Assuming Theorem \ref{RvMtheorem}, the following fractional decomposition formula
\begin{eqnarray*}
\dfrac{\Delta'_\mathcal{L}(s)}{\Delta_\mathcal{L}(s)-a}
=\sum\limits_{|t-\gamma_a|\leq 1/\log\log t}\dfrac{1}{s-\delta_a}+O\left(\log t\right)
\end{eqnarray*}
holds for any $-1\leq\sigma\leq 2$ and $t\gg1$ sufficiently large.
\end{Lem}

\begin{proof}
We employ a classical result which can be found in \cite[\S3.9, Lemma a]{tit}:

{\it If $f$ is regular and
$$\max\limits_{|s-s_0|\leq r}\left|\dfrac{f(s)}{f(s_0)}\right|<e^M$$
for some $M>1$ and $r>0$, then
$$\dfrac{f'}{f}(s)=\sum\limits_{|\rho-s_0|\leq r/2}\dfrac{1}{s-\rho}+O\left(\dfrac{M}{r}\right)$$
for any complex number $s$ such that $|s-s_0|\leq r/4$.}

We take here $$f(s)=\Delta_\mathcal{L}(s)-a,\qquad s_0=\dfrac{1}{2}+\dfrac{1}{\log\log y}+iy,\qquad r=\dfrac{8}{\log\log y},$$ where $y>0$.
If $y>1$ is sufficiently large, then \eqref{Deltasymp} yields that
$$\max\limits_{|s-s_0|\leq 8/\log\log y}\left|\dfrac{\Delta_\mathcal{L}(s)-a}{\Delta_\mathcal{L}(s_0)-a}\right|\ll \dfrac{|a|+y^{7d_\mathcal{L}/\log\log y}}{|a|-y^{-d_\mathcal{L}/\log\log y}}\ll \exp\left(\dfrac{7d_\mathcal{L}\log y}{\log\log y}\right),$$
and thus
$$\dfrac{\Delta'_\mathcal{L}(s)}{\Delta_\mathcal{L}(s)-a}=\sum\limits_{|\delta_a-s_0|\leq 4/\log\log y}\dfrac{1}{s-\delta_a}+O\left(\log y\right)$$
for any complex number $s$ such that $|s-s_0|\leq 2/\log\log y$.
In particular, this holds true for $t=y$ and any $\left|\sigma-1/2\right|\leq 1/\log\log t$.
Therefore,
\begin{eqnarray*}
\dfrac{\Delta'_\mathcal{L}(s)}{\Delta_\mathcal{L}(s)-a}
&=&\sum\limits_{|t-\gamma_a|\leq 1/\log\log t}\dfrac{1}{s-\delta_a}+\mathop{\sum\limits_{|\delta_a-s_0|\leq 4/\log\log t}}_{1/\log\log t<|t-\gamma_a|}\dfrac{1}{s-\delta_a}+O\left(\log t\right)
\end{eqnarray*}
for any $\left|\sigma-1/2\right|<1/\log\log t$ and $t\geq t_0\geq3$.
Now Theorem \ref{RvMtheorem} implies that the second sum above has $O(\log t/\log\log t)$ terms, each of which is $O(\log\log t)$.
Lastly, if $\sigma\in[-1,2]$ is such that $\left|\sigma-1/2\right|\geq1/\log\log t$, we bound $\Delta'_\mathcal{L}(s)/\left(\Delta_\mathcal{L}(s)-a\right)$ trivially by $O\left(\log t\right)$ as follows from Lemma \ref{StirlingsDelta} and \eqref{Deltasymp}.
Indeed, if $1/2+1/\log\log t\leq \sigma\leq2$, then
$$\dfrac{\Delta'_\mathcal{L}(s)}{\Delta_\mathcal{L}(s)-a}=\dfrac{\Delta'_\mathcal{L}(s)}{\Delta_\mathcal{L}(s)}\left(1-\dfrac{a}{\Delta_\mathcal{L}(s)}\right)^{-1}\ll\log t\left|\dfrac{|a|}{\left(\lambda Q^2t^{d_\mathcal{L}}\right)^{1/2-\sigma}}-1\right|^{-1}\ll\log t.$$
Similarly, we obtain the same bound when $-1\leq\sigma\leq1/2-1/\log\log t$.
\end{proof}

\vspace{3cm}
\noindent
Athanasios Sourmelidis\\
Institute of Analysis and Number Theory, TU Graz\\
Steyrergasse 30, 8010 Graz, Austria\\
sourmelidis@math.tugraz.at
\medskip

\noindent
J\"orn Steuding\\
Department of Mathematics, University of W\"urzburg\\
Emil Fischer-Str. 40, 97\,074 W\"urzburg, Germany\\
steuding@mathematik.uni-wuerzburg.de
\medskip

\noindent
Ade Irma Suriajaya\\
Faculty of Mathematics, Kyushu University \\
744 Motooka, Nishi-ku, Fukuoka 819-0395, Japan\\
adeirmasuriajaya@math.kyushu-u.ac.jp


\newpage
\begin{thebibliography}{99}

\bibitem{balgon}{\sc S. Baluyot, S. Gonek}, Explicit formulae and discrepancy estimates for a-points of the Riemann zeta-function, {\it Pacific J. Math} {\bf 303} (2019), 47--71.

\bibitem{berndt}{\sc B. Berndt}, On the zeros of a class of Dirichlet series I., {\it Illinois J. Math.} {\bf 14} (1970), 244--258.

\bibitem{bourg}{\sc J. Bourgain}, Decoupling, exponential sums and the Riemann zeta function, {\it J. Amer. Math. Soc.} {\bf 30} (2017), 205--224.

\bibitem{burckel}{\sc R.B. Burckel}, {\it An Introduction to Classical Complex Ananlysis}, vol. I, Birkh\"auser, 1979.

\bibitem{congo}{\sc J.B. Conrey, A. Ghosh}, On the Selberg class of Dirichlet series: small degrees, {\it Duke Math. J.} {\bf 72} (1993), 673--693.

\bibitem{dixsch}{\sc R.D. Dixon, L. Schoenfeld}, The size of the Riemann zeta-function at places symmetric with respect to the point $\frac{1}{ 2}$, {\it Duke Math. J.} {\bf 33} (1966), no. 2, 291--292.

\bibitem{gonek2}{\sc S.M. Gonek}, An explicit formula of Landau and its applications to the theory of the zeta-function. A tribute to Emil Grosswald: number theory and related analysis, Amer. Math. Soc., Providence, RI, {\it Contemp. Math.} {\bf 143} (1993), 395--413.

\bibitem{hamburger}{\sc H. Hamburger}, \"Uber die Riemannsche Funktionalgleichung der $\zeta$-Funktion. I. {\it Math. Z.} {\bf 10} (1921), 240--254.

\bibitem{hecke}{\sc E. Hecke}, \"Uber die Bestimmung Dirichletscher Reihen durch ihre Funktionalgleichung, {\it Math. Ann.} {\bf 112} (1936), 664--699.

\bibitem{ivic}{\sc A. Ivi\'c}, {\it The Riemann zeta-function. Theory and applications.}, Dover Publications, Inc., Mineola, NY, 2003.

\bibitem{julia}{\sc G. Julia}, Sur quelques propri\'et\'es nouvelles des fonctions enti\`eres ou m\'eromorhes (liere M\'emoire), {\it Ann. Sci. \'Ecole Norm. Sup.} {\bf 36} (1919), 93--125.

\bibitem{kac}{\sc J. Kaczorowski, G. Molteni, A. Perelli, J. Steuding, J. Wolfart}, Hecke's theory and the Selberg class, {\it Funct. Approx. Comment. Math.} {\bf 35} (2006), 183--193.

\bibitem{kape}{\sc J. Kaczorowski, A. Perelli}, On the structure of the Selberg class. I. $0\leq d\leq1$, {\it Acta Math.} {\bf 182} (1999), no. 2, 207--241. 

\bibitem{kape1}{\sc J. Kaczorowski, A. Perelli}, On the structure of the Selberg class. II. Invariants and conjectures, {\it J. Reine Angew. Math.} {\bf 524} (2000), 73--96.

\bibitem{kape2}{\sc J. Kaczorowski, A. Perelli}, On the structure of the Selberg class. VII. $1<d<2$, {\it Annals of Math.} {\bf 173} (2011), 1397--1441. 

\bibitem{kape3}{\sc J. Kaczorowski, A. Perelli}, Classification of $L$-functions of degree $2$ and conductor $1$, arXiv:2009.12329.

\bibitem{kapsteu}{\sc J. Kalpokas, J. Steuding}, On the value-distribution of the Riemann zeta-function on the critical line, {\it Mosc. J. Comb. Number Theory} {\bf 1} (2011), 26--42. 

\bibitem{land}{\sc E. Landau}, \"Uber die Nullstellen der Zetafunktion, {\it Math. Ann.} {\bf 71} (1911), 548--564.

\bibitem{lang}{\sc S. Lang}, {\it Complex Analysis}, Springer, 1999.

\bibitem{levi}{\sc N. Levinson}, Almost all roots of $\zeta(s)=a$ are arbitrarily close to $\sigma={\textstyle\frac{1}{2}}$, {\it Proc. Nat. Acad. Sci. USA} {\bf 72} (1975), 1322--1324.

\bibitem{nico}{\sc N. Oswald}, On a relation between modular functions and Dirichlet Series (found in the estate of Adolf Hurwitz, {\it Arch. Hist. Exact Sci.} {\bf 71} (2017), 345--361.

\bibitem{picard}{\sc \'E. Picard}, Sur les fonctions analytiques uniformes dans le voisinage d'un point singulier essentiel, {\it C.R. Acad. Sci. Paris} {\bf 89} (1879), 745--747.

\bibitem{reh}{\sc M. Rehberg}, A discrepancy estimate for the a-points of the Riemann zeta function, {\it Analytic and probabilistic methods in number theory. Proceedings of the sixth international conference}, Palanga, Lithuania, September 11--17, 2016, 165--178, Vilnius University Publishing House, Vilnius, 2017.

\bibitem{sel} {\sc A. Selberg}, {\it Old and new conjectures and results about a class of Dirichlet series}, in Proceedings of the Amalfi Conference on Analytic Number Theory (Maiori, 1989), vol. 2, 1992, 47--63.

\bibitem{soursteusur}{\sc A. Sourmelidis, J. Steuding, A.I. Suriajaya}, Dirichlet Series with Periodic Coefficients and their Value-Distribution Near the Critical Line, arXiv:2007.14008, submitted.

\bibitem{spira} {\sc R. Spira}, {An inequality for the Riemann zeta function.}, {\it Duke Math. J.} {\bf 11} (1965), 247--250. 

\bibitem{steudi}{\sc J. Steuding}, {\it Value-distribution of $L$-functions}, Lecture Notes in Mathematics {\bf 1877}, Springer 2007.

\bibitem{steu1}{\sc J. Steuding}, One hundred years uniform distribution modulo one and recent applications to Riemann’s zeta-function, in: {\it Topics in mathematical analysis and applications}, Th.M. Rassias (ed.) et al., Springer Optimization and Its Applications 94, 659--698 (2014).

\bibitem{steusur}{\sc J. Steuding, A.I. Suriajaya}, Value-distribution of the Riemann zeta-function along its Julia lines, {\it Comp. Methods Function Theory} {\bf 20} (2020), doi: 10.1007/s40315-020-00316-x.

\bibitem{titchm}{\sc E.C. Titchmarsh}, On van der Corput's method and the zeta-function of Riemann, {\it Quarterly J. Math.} {\bf 5} (1935), 98--105.

\bibitem{titch}{\sc E.C. Titchmarsh}, {\it The Theory of Functions}, 2nd ed. Oxford University Press, 1939.

\bibitem{tit}{\sc E.C. Titchmarsh}, {\it The Theory of the Riemann zeta-function}, 2nd ed., revised by {\sc D.R. Heath-Brown}, Oxford University Press, 1986.

\bibitem{weil}{\sc A. Weil}, Prehistory of the zeta-function, in: {\it Number Theory, Trace Formulas and Discrete Groups}, Oslo 1987, ed. by E.A. Aubert et al., Academic, Boston 1989, 1--9. 

\bibitem{weyl}{\sc H. Weyl}, \"Uber die Gleichverteilung von Zahlen mod. Eins, {\it Math. Ann.} {\bf 77} (1916), 313--352.

\end{thebibliography}
\end{document}